\documentclass{article}
\usepackage{amssymb}
\usepackage{amsmath}
\usepackage{amsfonts}

\newtheorem{theorem}{Theorem}

\newtheorem{corollary}[theorem]{Corollary}

\newtheorem{definition}[theorem]{Definition}
\newtheorem{example}[theorem]{Example}

\newtheorem{lemma}[theorem]{Lemma}

\newtheorem{proposition}[theorem]{Proposition}
\newtheorem{remark}[theorem]{Remark}

\newenvironment{proof}[1][Proof]{\noindent\textbf{#1.} }{\ \rule{0.5em}{0.5em}}
\input{tcilatex}
\begin{document}

\begin{center}
\bigskip

\bigskip

\bigskip

\bigskip

\bigskip

\bigskip

\bigskip

\bigskip

\textbf{ON PRODUCT AND GOLDEN STRUCTURES}

\textbf{AND}

\textbf{HARMONICITY}

\bigskip

Sadettin ERDEM

\bigskip

{\small Mathematics Department}

{\small Faculty of Science}

{\small Anadolu University}

{\small Eski\c{s}ehir-Turkey}\newpage

\textbf{ON PRODUCT AND GOLDEN STRUCTURES}

\textbf{AND}

\textbf{HARMONICITY}

.

Sadettin ERDEM

{\small Mathematics Department}

{\small Faculty of Science}

{\small Anadolu University}

{\small Eski\c{s}ehir-Turkey}

\ 
\end{center}

\textbf{Abstract:} {\small In this work, almost product and almost golden
structures are studied. Conditions for those structures being Integrable and
parallel are investigated. Also harmonicity of a map between almost pruduct
or almost golden manifolds with pure or hyperbolic metric is discussed under
certain conditions.}

\ 


\ 


\ 

\section{1. Introduction}

Let $\ \ M^{n}$ \ be a smooth manifold of dimension \ $n$ \ with a \ $(1,1)$%
-tensor field $\ \ \varphi $ \ \ of rank \ \ $n.$\ Then $\left( see\ \left[ 
\mathbf{3,\ 4,\ 5,\ 8,\ 12,\ 13}\right] \right) $ the pair \ $\left(
M,\varphi \right) $ \ is called \textit{polynomial} \textit{manifold }%
provided \ $\mathcal{U}\left( \varphi \right) =0$ \ for some polynomial $\ \ 
\mathcal{U}\left( x\right) $ \ \ over the field of real numbers \ $%
\mathbb{R}
.$ \ In particular, \ \ $\varphi $ \ \ and \ $\left( M,\varphi \right) $ \ \
are respectively called

$i)$ \ \ \textit{metallic} \textit{structure} \ and \ \textit{metallic} 
\textit{manifold }if \ \ $\mathcal{U}\left( x\right) =x^{2}-\eta x-\delta $
\ \ for some positive integers \ $\eta $ $\ $and $\ \delta ,$\ \ \ so that \
\ $\mathcal{U}\left( \varphi \right) =\varphi ^{2}-\eta \varphi -\delta =0.$

$ii)$ \ \textit{almost complex} \textit{structure} \ and \textit{almost
complex manifold \ }if \ \ $\mathcal{U}\left( x\right) =x^{2}+1,$ \ \ so
that \ \ $\mathcal{U}\left( \varphi \right) =\varphi ^{2}+I=0.$

$iii)$ \ \textit{\ almost} \textit{product structure} \ and \ \textit{almost}
\textit{product manifold }if \ $\mathcal{U}\left( x\right) =x^{2}-1.$ \ \ In
this case we reserve the letter \ $P$ \ \ for \ $\varphi .$ $\ \ $Thus $\ \ 
\mathcal{U}\left( P\right) =P^{2}-I=0.$

$iv)$ \ \ \textit{almost} \textit{golden} \textit{structure} \ and \ \textit{%
almost} \textit{golden} \textit{manifold }if \ $\mathcal{U}\left( x\right)
=x^{2}-x-1.$ \ \ In this case we reserve the letter \ $G$ \ \ for \ $\varphi
.$ $\ \ $Thus $\ $\ $\mathcal{U}\left( G\right) =G^{2}-G-I=0.$

Where \ $I$ $\ $denotes$\mathcal{\ }$the identity tensor field.$\mathcal{\ }%
\mathcal{\ }$

Note here that an almost golden manifold \ \ $\left( M,G\right) $ \ is in
fact a metallic manifold with \ \ $\eta =\delta =1.$ Golden structure has
been catching more attention of many geometers \ $\left( see:\text{for
example}\ \left[ \mathbf{3,\ 4,\ 12,\ 13}\right] \right) $ \ in the last few
years as it is closely related to the golden ratio which plays an important
role in various disciplines such as physics, topology, probability, field
theory etc. ($see$ \ $\left[ \mathbf{3,\ 4}\right] $ \ and the references
therein )

In this work, we have dealt with almost product and almost golden structures
simultaneously as one can be obtained from the other, and provided the
following results besides some side ones:

Let \ \ $\varphi $ \ denote either almost product structure \ $P$ \ or \
almost golden structure \ $G.$ \ To emphasize this we shall be writing \ \ $%
\varphi \left( =P,G\right) .$ On an almost product or an almost golden
manifold \ $\left( M,\ h,\ \varphi \left( =P,G\right) \right) $ \ with pure
or hyperbolic metric \ $h,$\ $\left( \text{Definition }\left( 2.1\right)
\right) $.\ 

$1)$ \ By analogy with the result for the paracomplex case, \ we introduce a
condition \ $\varphi \left( \ast \right) $ \ (\textit{see page 12, just
before Proposition \ (2.5)}) which, together with the integrability
condition of \ $\varphi ,$ \ guarantees that \ $\varphi $ \ is parallel, $%
\left( \text{Proposition }\left( 2.5\right) \right) $\ 

$2)$ \ For the bilinear operator \ $S_{\varphi }:\Gamma \left( TM\right)
\times \Gamma \left( TM\right) \rightarrow \Gamma \left( TM\right) $ \ ($%
see: $ right after Definition $\left( 2.4\right) $) it is shown that
vanishing of $S_{\varphi }$ \ is equivalent to that of $\varphi $ \ being
parallel$,$ $\left( \text{Proposition }\left( 2.6\right) \right) $, \ unlike
the case in which the metric \ $h$ \ is hyperbolic, vanishing of \ $%
S_{\varphi }$ \ does not imply that $\varphi $ \ is parallel.\ Instead, it
provides a bigger class whose members are called \ quasi para-Hermitian
manifolds, quasi golden-Hermitian manifolds \ $\left( \text{Definition }%
\left( 2.5\right) \right) $.

$3)$ \ We introduced a subclass of \ $\left( M,\ h,\ \varphi \left(
=P,G\right) \right) ,$ \ namely, a class of semi decomposable product (or
golden ) Riemannian manifolds $\left( \text{Definition }\left( 2.4\right)
\right) $ and that used later on for the harmonicity of certain map from, $%
\left( \text{Theorems }\left( 3.1\right) \text{ }\&\ \left( 3.2\right)
\right) $.

$4)$ \ By analogy with the concept of \textit{anti-paraholomorphic map}, a
concept of \textit{antigolden map} is introduced $\left( \text{Definition }%
\left( 3.\ 2\right) \right) $ and that later it is used for its harmonicity, 
$\left( \text{Theorems }\left( 3.1\right) \text{ }\&\ \left( 3.2\right)
\right) $.

$5)$ \ It is shown that being a golden ($resp:$ paraholomorphic) map of an
almost golden ($resp:$ almost product) manifold with a pure metric is no way
sufficient for its harmonicity where as it is sufficient \ when the metric
is hyperbolic. However, on the same line, an alternative result is provided,
\ $\left( \text{Theorem }\left( 3.1\right) \right) $.

$6)$ \ Finally,$\left( \text{Theorems }\left( 3.1\right) \text{ }\&\ \left(
3.2\right) \right) ,$ for a non-constant map \ $F:\left( M,\ h,\ \varphi
\left( =P,G\right) \right) \rightarrow \left( N,\ g,\ \varphi \left(
=Q,K\right) \right) ,$ \ where \ $h$ \ and \ $g$ \ are hyperbolic, the
harmonicity result given for \ $\pm \left( P,Q\right) $-paraholomorphic map
\ $F$,\ $\left( \left[ \mathbf{2,\ 7,\ 11}\right] \right) ,$ \ is extended
to the cases where

\begin{itemize}
\item $F$ $\ $is $\ \pm \left( P,Q\right) $-paraholomorphic and \ $h$ \ is
hyperbolic,\ $g$ \ is pure.\ 

\item $F$ $\ $is $\ \pm \left( G,K\right) $-golden and \ $h$ \ is
hyperbolic,\ $g$ \ is pure or hyperbolic.\ 
\end{itemize}

$7)$ \ Overall, we have managed so far to express the results involving
almost golden structures in terms of almost product structures.

\section{2. Definitions and some basic results}

The structure \ $\varphi \left( =P,G\right) $ \ \ on \ $M^{n}$ \ has two
distinct real eigenvalues, namely; $\ k\ $\ and $\ \overline{k}.$ \ Let
denote the corresponding eigendistributions by \ $\ \mathcal{E}_{\left(
k\right) },\ $and $\ \ \mathcal{E}_{\left( \overline{k}\right) }.$ \ 

Note that, \ $(see$ \ $\left[ \mathbf{2,\ 3,\ 4,\ 5,\ 8,\ 12,\ 13}\right] ),$

$1)$ \ \ $\varphi :TM\mathcal{\rightarrow }TM$ \ \ is an isomorphism.

$2)$ \ \ $TM=\mathcal{E}_{\left( k\right) }\oplus \mathcal{E}_{\left( 
\overline{k}\right) }$ \ \ .

$3)$ \ \ For an almost product manifold \ $\left( M,P\right) $ \ we have \ 

\begin{itemize}
\item $k=1\ \ \ $and $\ \ \overline{k}=-1.$ \ 

\item $P^{2}\left( X\right) =X,\ \ \ \ \forall \ \ X\in \Gamma \left(
TM\right) $
\end{itemize}

$4)$ \ \ For an almost golden manifold \ $\left( M,\ G\right) $ \ we have

\begin{itemize}
\item $k=\frac{1}{2}\left( 1+\sqrt{5}\right) \ \ \ $and $\ \ \overline{k}=%
\frac{1}{2}\left( 1-\sqrt{5}\right) .$
\end{itemize}

Through out this work we shall be setting $\ \ \sigma =\frac{1}{2}\left( 1+%
\sqrt{5}\right) $ $\ \ $and $\ \ \bar{\sigma}=\frac{1}{2}\left( 1-\sqrt{5}%
\right) .$ \ Observe that,%
\begin{equation*}
\sigma ^{2}=\sigma +1,\ \ \ \ \bar{\sigma}^{2}=\bar{\sigma}+1\text{ \ \ \
and \ \ \ }\sigma \bar{\sigma}=-1.
\end{equation*}

\begin{itemize}
\item $G^{2}\left( X\right) =GX+X,\ \ \ \ \forall \ \ X\in \Gamma \left(
TM\right) .$
\end{itemize}

$\bigskip 5)$

\begin{itemize}
\item for every almost product structure \ \ $P$, \ \ define a $P$\textit{%
-associated }$(1,1)$-\textit{tensor field} \ \ $G_{P}=\mathcal{G}$ \ \ by%
\begin{equation*}
G_{p}=\mathcal{G}=\frac{1}{2}\left( I+\sqrt{5}P\right)
\end{equation*}

\item for every almost golden structure \ \ $G$, \ \ define a $G$-associated 
$(1,1)$-tensor field \ \ $P_{{\small G}}=\mathcal{R}$ \ \ by 
\begin{equation*}
P_{{\small G}}=\mathcal{R}=\frac{1}{\sqrt{5}}\left( 2G-I\right)
\end{equation*}
\end{itemize}

Note that,

$i)\ $\ for every almost product structure \ $P$ \ on \ $M,$ \ the
corresponding \ $G_{p}=\mathcal{G}$ \ is an almost golden structure on \ $M$
\ and \ therefore it will be called \ $P$\textit{-associated} \textit{almost
golden structure.}

$ii)$\ \ for every almost golden structure \ $G$ \ on $\ M,$\ the
corresponding \ $P_{{\small G}}=\mathcal{R}$ \ is an almost product
structure on \ $M$ \ and \ therefore it will be called \ $G$\textit{%
-associated} \textit{almost product structure.}

$iii)$ \ we have%
\begin{equation*}
\mathcal{E}_{\left( \sigma \right) }^{\mathcal{G}}=\mathcal{E}_{\left(
1\right) }^{{\small P}}\ \ \ \ \ \ \ \ \ \ \ \ \ and\ \ \ \ \ \ \ \ \ \ \ \ 
\mathcal{E}_{\left( \bar{\sigma}\right) }^{\mathcal{G}}=\mathcal{E}_{\left(
-1\right) }^{{\small P}}
\end{equation*}

$iv)$\ \ there is a one-to-one correspondence between the set of \ all
almost product structures and the set of all almost golden structures on a
manifold \ \ $M.$ \ We shall be calling \ \ $\left\{ P,\ G_{_{P}}\right\} $
\ (or \ \ $\left\{ G,\ P_{_{G}}\right\} $ )\ \textit{an associated pair }or%
\textit{\ \ a twin pair.} we also say that \ $\left\{ P,\ G_{_{P}}\right\} $
\ (or \ \ $\left\{ G,\ P_{_{G}}\right\} $ ) \ are \textit{twins}. It is easy
to see that \ for a given pair of twin structures $\ \left\{ P,\
G_{_{P}}\right\} ,$ \ the $\ \ G_{_{P}}$-associated almost product structure
is equal to \ $P$, \ \ that is, \ \ \ $\ $%
\begin{equation*}
P_{\left( {\small G}_{P}\right) }=P
\end{equation*}

Similarly, for a twin pair \ $\left\{ G,\ P_{_{G}}\right\} ,$ \ the $\ \
P_{_{G}}$-associated almost golden structure is equal to \ $G$, \ \ that is,%
\begin{equation*}
G_{\left( P_{G}\right) }=G
\end{equation*}

$v)$\ \ If \ $P$ \ is an almost product structure on \ $M$ \ then \ \ \ $%
\widehat{P}=-P$ \ is also an almost product structure on \ $M.$ \ Observe
that \ \ $P$ \ and \ $\widehat{P}$ \ have the same eigenvalues $\ 1$ \ and \ 
$-1$. \ However for their corresponding eigendistributions we have 
\begin{equation*}
\text{\ \ }\mathcal{E}_{\left( 1\right) }^{P}=\mathcal{E}_{\left( -1\right)
}^{\widehat{P}}\ \text{\ \ \ \ }\ \text{and}\ \ \ \ \ \text{\ }\mathcal{E}%
_{\left( 1\right) }^{\widehat{P}}=\mathcal{E}_{\left( -1\right) }^{P}\ \ \ \ 
\end{equation*}%
We shall be calling \ $\widehat{P},$ \ \textit{the conjugate almost product
structure of} \ $P$ \ or \ \ \textit{the }$P$-\textit{conjugate almost
product structure}

$vi)$\ \ \ If \ $G$ \ is an almost golden structure on \ $M$ \ then \ \ \ \ $%
\widehat{G}=I-G$ \ is also an almost golden structure on \ $M.$ \ Observe
that \ \ $G$ \ and \ $\widehat{G}$ \ have the same eigenvalues $\ \sigma $ \
and \ $\bar{\sigma}$. \ However for their corresponding eigendistributions
we have%
\begin{equation*}
\mathcal{E}_{\left( \sigma \right) }^{G}=\mathcal{E}_{\left( \bar{\sigma}%
\right) }^{\widehat{G}},\ \ \ \ \ \ \ \ \ \ \text{and}\ \ \ \ \ \ \ \ \ \ \
\ \ \ \ \text{\ }\mathcal{E}_{\left( \bar{\sigma}\right) }^{G}=\mathcal{E}%
_{\left( \sigma \right) }^{\widehat{G}}
\end{equation*}%
We shall be calling \ $\widehat{G},$ \ \textit{the conjugate almost golden
structure of} \ $G$ \ or \ \ \textit{the }$G$-\textit{conjugate almost
golden structure.}

$vii)$ \ \ If $\ \left\{ P,\ G\right\} $ \ \ is a twin pair then $\ \widehat{%
G}=G_{\widehat{p}}=\frac{1}{2}\left( I+\sqrt{5}\widehat{\left(
P_{_{G}}\right) }\right) =\frac{1}{2}\left( I-\sqrt{5}P_{_{G}}\right) ,$\
that is, $\ \left\{ \widehat{P},\ \widehat{G}\right\} $ \ is also a twin
pair. Conversely, if \ $\left\{ \widehat{P},\ \widehat{G}\right\} $ \ is a
twin pair then $\ G=\frac{1}{2}\left( I+\sqrt{5}P\right) ,\ \ \ $that is, \ $%
\ \left\{ P,\ G\right\} $ \ \ is also a twin pair

$viii)$ \ If $\ \ \left\{ \widehat{P},\ \widehat{G}\right\} $ \ is a twin
pair then $\ $%
\begin{equation*}
\mathcal{E}_{\left( 1\right) }^{P}=\mathcal{E}_{\left( -1\right) }^{\widehat{%
P}}=\mathcal{E}_{\left( \sigma \right) }^{G}=\mathcal{E}_{\left( \bar{\sigma}%
\right) }^{\widehat{G}}\ \ \ \ \ \text{and}\ \ \ \ \mathcal{E}_{\left(
1\right) }^{\widehat{P}}=\mathcal{E}_{\left( -1\right) }^{P}=\mathcal{E}%
_{\left( \bar{\sigma}\right) }^{G}=\mathcal{E}_{\left( \sigma \right) }^{%
\widehat{G}}.
\end{equation*}

$\blacksquare $

An almost product manifold\textit{\ \ }$\left( M,\ P\right) $ \ is called \ 
\textit{an almost\ paracomplex manifold} if the eigendistributions \ $\ 
\mathcal{E}_{\left( 1\right) },\ $and $\ \ \mathcal{E}_{\left( -1\right) }$
\ are of the same rank, $\left( \left[ \mathbf{2,\ 5}\right] \right) $. An
almost golden manifold\textit{\ \ }$\left( M,\ G\right) $ \ is called \ 
\textit{an almost\ para-golden manifold} if the eigendistributions \ $\ 
\mathcal{E}_{\left( \sigma \right) },\ $and $\ \ \mathcal{E}_{\left( \bar{%
\sigma}\right) }$ \ are of the same rank. \ It clear from their definitions
that \textit{an }almost\ paracomplex manifold \textit{\ }$\left( M,\
P\right) $ \ and an\ almost\ para-golden manifold\textit{\ \ }$\left( M,\
G\right) $ \ are necessarily of even dimensions.

\begin{definition}
(2.1/A): \ \ Let $M$ \ \ be a smooth manifold together with a \ $(1,1)$ \
tensor field \ $\varphi \left( =P,G\right) $ \ and a Riemannian metric \ $h$
\ satisfying%
\begin{equation}
h\left( \varphi X,\ Y\right) =\ h\left( X,\ \varphi Y\right) ;\ \ \ \ \
\forall \ X,\ Y\in \Gamma \left( TM\right) .  \tag{$\left( \ast \right) $}
\end{equation}
\end{definition}

\textit{Then}

$\ i)$ \ \ $\left( M,\ h,\ P\right) $ \ \ is called \textit{\ almost product
Riemannian manifold, }$\left[ \mathbf{8}\right] $\textit{.}

$ii)$ \ $\ \ \left( M,\ h,\ G\right) $ \ \ is called \textit{almost golden
Riemannian manifold, }$\left( \left[ \mathbf{4,\ 8}\right] \right) .\ $%
\textit{.}

We refer the condition$\ \ \left( \ast \right) \ \ $as \textit{the
compatibility of \ \ }$h$\textit{\ \ \ and \ \ }$\varphi .$\ \ We also say \
"$\ h$ \ \textit{is pure with respect to} \ \ $\varphi \ "$\ \ if $\ \ h$%
\textit{\ \ \ and \ \ }$\varphi $ \ \ are compatible, and call \ $h$ \ 
\textit{pure metric (with respect to \ \ }$\varphi $\textit{).} \ \ Note
here that the eigendistributions $\mathcal{\ \ E}_{\left( k\right) }\ \ $and 
$\ \ \mathcal{E}_{\left( \bar{k}\right) }$ \ \ are $h$-orthogonal.

$iii)$ \ \textit{\ }An almost product Riemannian manifold\textit{\ \ \ }$%
\left( M,\ h,\ P\right) $ \ and \ its metric \ $h$ \ \ are also called 
\textit{almost }\ \textbf{B}\textit{-manifold} and \textbf{B}\textit{-metric}
respectively if the eigendistributions \ $\ \mathcal{E}_{\left( 1\right) }\
\ $and $\ \ \mathcal{E}_{\left( -1\right) }$ \ are of the same rank, $\left[ 
\mathbf{12}\right] $.

$iv)$ \ \textit{\ }An almost golden Riemannian manifold\textit{\ \ \ }$%
\left( M,\ h,\ P\right) $ \ is also called \textit{almost para-golden
Riemannian manifold}\ if the eigendistributions \ $\ \mathcal{E}_{\left(
\sigma \right) },\ $and $\ \ \mathcal{E}_{\left( \bar{\sigma}\right) }$ \
are of the same rank.

\begin{definition}
(2.1/B): \ \ Let \ $M$ \ \ be a smooth manifold together with a \ $(1,1)$ \
tensor field \ $\varphi \left( =P,\ G\right) $ \ and a nondegenerate metric
\ $h$ \ satisfying%
\begin{equation}
h\left( \varphi X,\ Y\right) =\ h\left( X,\ \widehat{\varphi }Y\right) ;\ \
\ \ \ \forall \ X,\ Y\in \Gamma \left( TM\right) . 
\tag{$\left(
\ast
\ast
\right) $}
\end{equation}
\end{definition}

Then

$i)$ \ \ $\left( M,\ h,\ P\right) $ \ \ is called \textit{almost
para-Hermitian manifold, }$\left[ \mathbf{2}\right] .$\textit{\ \ }

$ii)$ \ $\ \left( M,\ h,\ G\right) $ \ \ is called \textit{almost
golden-Hermitian manifold}.

In this case, we refer the conditions $\ \left( \ast \ast \right) \ \ \ $as 
\textit{the hyperbolic compatibility of \ \ }$h$\textit{\ \ \ and \ \ }$%
\varphi .$ \ \ We also say \ "$h$ \ is \textit{hyperbolic with respect to }\
\ $\varphi "\ $\ \ if $\ \ h$\textit{\ \ \ and \ \ }$\varphi $ \ \ are
hyperbolic compatible, and call \ $h$ \ \textit{hyperbolic metric (with
respect to \ \ }$\varphi $\textit{)} \ \ \ \ \ \ $\blacksquare $

Note here that the\ hyperbolic case differs from the pure one. To be precise:

On a manifold \ \ $\left( M,\ h,\ \varphi \right) $ \ \ with a \ hyperbolic
metric $\ h$ \ (with respect to \ \ $\varphi $) one has,

$1)$ \ \ $h\left( PX,\ Y\right) =h\left( X,\ \widehat{P}Y\right) =-h\left(
X,\ PY\right) ;\ \ \ \ \ \forall \ X,\ Y\in \Gamma \left( TM\right) .$ \
Therefore we have \ \ \ 
\begin{equation*}
h\left( PX,\ X\right) =0,\ \ \ \ \forall \ X\in \Gamma \left( TM\right)
\end{equation*}%
unlike the pure case where, for example, 
\begin{equation*}
h\left( PX,\ X\right) =h\left( X,\ X\right) ,\ \ \ \ \forall \ X\in \Gamma
\left( \mathcal{E}_{\left( 1\right) }^{P}\right)
\end{equation*}

$2)$ \ \ $h\left( GX,\ Y\right) =h\left( X,\ \widehat{G}Y\right) ;\ \ \
\forall \ X,\ Y\in \Gamma \left( \mathcal{D}\right) .$ Therefore we have \ 
\begin{equation*}
2h\left( GX,\ X\right) =h\left( X,\ X\right) =2h\left( \widehat{G}X,\
X\right) ,\ \ \ \ \ \ \ \forall \ X\in \Gamma \left( TM\right) .
\end{equation*}

$3)$ \ \ $h\left( X,\ Y\right) =0;\ \ \ \ \forall \ X,\ Y\in \Gamma \left( 
\mathcal{E}_{\left( k\right) }^{\varphi }\right) \ \ \ \ \ $or $\ \ \ \ \
\forall \ X,\ Y\in \Gamma \left( \mathcal{E}_{\left( \bar{k}\right)
}^{\varphi }\right) .$ \ That is, hyperbolic metric\textit{\ }\ $h$ \ is
null on the eigendistributions \ $\mathcal{E}_{\left( k\right) }^{\varphi }\
\ \ $and $\ \ \ \mathcal{E}_{\left( \bar{k}\right) }^{\varphi }$ \ (and
therefore the hyperbolic metric is necessarily semi-Riemannian where as the
pure metric is taken to be Riemannian.)

Indeed, let \ \ $X,\ Y\in \Gamma \left( \mathcal{E}_{\left( k\right)
}^{\varphi }\right) $ \ \ then \ \ \ $h\left( \varphi X,\ Y\right) =kh\left(
X,\ Y\right) $ $\ \ $and $\ \ \ h\left( X,\ \widehat{\varphi }Y\right) =\bar{%
k}h\left( X,\ Y\right) .\ \ \ $On the other hand ,\ \ $h\left( \varphi X,\
Y\right) =h\left( X,\ \widehat{\varphi }Y\right) $ $\ \ $since $\ h\ \ $is
hyperbolic. $\ $So$\ \ \ kh\left( X,\ Y\right) =\bar{k}h\left( X,\ Y\right)
, $ \ \ \ which gives \ \ $\left( k-\bar{k}\right) h\left( X,\ Y\right) =0,$
\ \ so that \ \ \ $h\left( X,\ Y\right) =0;\ \ \forall \ X,\ Y\in \Gamma
\left( \mathcal{E}_{\left( k\right) }^{\varphi }\right) .$ By\ the same
argument we get \ \ $\forall \ X,\ Y\in \Gamma \left( \mathcal{E}_{\left( 
\bar{k}\right) }^{\varphi }\right) $\ \ \ \ \ \ \ \ $h\left( X,\ Y\right)
=0. $

\begin{lemma}
(2.2/A): $\left( \left[ \mathbf{2}\right] \right) $\ \ Let \ $\left( M,\ h,\
P\right) $\ $\ $be an \textit{almost para-Hermitian }manifold. Then
\end{lemma}

$\ i)$ \ \ $h$ \ \ is of signature \ \ $\left( m,\ m\right) $ \ on\ \ $TM$ \
for some \ $m.$

$ii)$ \ \ $rank\left( \mathcal{E}_{\left( 1\right) }^{P}\right) =rank\left( 
\mathcal{E}_{\left( -1\right) }^{P}\right) =m.$

\ \ \ \ \ \ \ \ \ \ \ \ \ \ \ \ \ \ \ \ \ \ \ \ \ \ \ \ \ \ \ \ \ \ \ \ \ \
\ \ \ \ \ \ \ \ \ \ \ \ \ \ \ \ \ \ \ \ \ \ \ \ \ \ \ \ \ \ \ \ \ \ \ \ \ \
\ \ \ \ \ \ \ \ \ \ \ \ \ \ \ \ \ \ \ \ \ $\blacksquare $

Having given an almost golden manifold \ \ $\left( M,\ h,\ G\right) $ \ with
a hyperbolic metric \ $h$, \ since \ $h$ \ is also hyperbolic with respect
to the product structure \ $P_{_{G}},$ \ by considering the almost
para-Hermitian manifold \ $\left( M,\ h,\ P_{_{G}}\right) $ \ and using the
above Lemma, we get:

\begin{lemma}
(2.2/B)$:$ \ \ Let \ \ $\left( M,\ h,\ G\right) $ \ \ be an \textit{almost
golden-Hermitian }manifold. Then
\end{lemma}

\ $i)$ \ $\ h$ \ \ is of signature \ \ $\left( m,\ m\right) $ \ on\ \ $TM$ \
for some \ $m.$

$ii)$ \ \ $rank\left( \mathcal{E}_{\left( \sigma \right) }^{G}\right)
=rank\left( \mathcal{E}_{\left( \bar{\sigma}\right) }^{G}\right) =rank\left( 
\mathcal{E}_{\left( 1\right) }^{\mathcal{R}}\right) =m,\ \ \ \ $where \ $%
\mathcal{R}=P_{_{G}}.$

\begin{proposition}
(2.1): \ Let \textit{an almost product }structure \ \ $P$ \ \ and an almost 
\textit{golden} structure \ \ $G$ \ \ form a twin pair \ $\left\{ P,\
G\right\} $ \ on a smooth manifold \ \ $M.\ \ $For a nondegenerate metric \ $%
h$ \ on \ \ $M$ \ the following statements are equivalent:
\end{proposition}

$\ i)$ \ \ $h$ \ \ is pure \ $\left[ resp:\text{ hyperbolic}\right] $ \ with
respect to \ \ $P.\ \ $

$ii)$ \ \ $h$ \ \ is pure \ $\left[ resp:\text{ hyperbolic}\right] $ \ with
respect to \ \ $\widehat{P}.\ \ $

$iii)$ \ \ $h$ \ \ is pure \ $\left[ resp:\text{ hyperbolic}\right] $ \ with
respect to \ \ $G.\ \ $

$iv)$ \ \ $h$ \ \ is pure \ $\left[ resp:\text{ hyperbolic}\right] $ \ with
respect to \ \ $\widehat{G}.$

\begin{proof}
We only be showing the equivalence of \ $(i)$ $\ $and $\ (iv)$ \ as the rest
of the cases follow by the similar argument:
\end{proof}

Assume \ $(i),$ $\ $then \ \ $\forall \ X,\ Y\in \Gamma \left( TM\right) $ $%
\ \ \ \ \ \ \ \ \ \ \ \ \ \ \ \ \ \ \ \ \ \ \ \ \ \ \ \ \ \ \ \ \ \ $%
\begin{equation*}
\begin{tabular}{ll}
$h\left( X,\ \widehat{G}Y\right) =$ & $h\left( X,\ \frac{1}{2}\left( I+\sqrt{%
5}\ \widehat{P}\right) Y\right) =\frac{1}{2}h\left( X,\ Y\right) +\frac{%
\sqrt{5}}{2}h\left( X,\ \ \widehat{P}Y\right) $ \\ 
& $=\frac{1}{2}h\left( X,\ Y\right) -\frac{\sqrt{5}}{2}h\left( X,\ PY\right)
=\frac{1}{2}h\left( X,\ Y\right) -\frac{\sqrt{5}}{2}h\left( PX,\ Y\right) $
\\ 
& $=\frac{1}{2}h\left( X,\ Y\right) +\frac{\sqrt{5}}{2}h\left( \ \widehat{P}%
X,\ Y\right) =h\left( \frac{1}{2}\left( I+\sqrt{5}\ \widehat{P}\right) X,\
Y\right) =h\left( \widehat{G}X,\ Y\right) .$%
\end{tabular}%
\end{equation*}

Next assume $\ (ii)$ $\ $then \ \ $\forall \ X,\ Y\in \Gamma \left(
TM\right) $%
\begin{equation*}
\begin{tabular}{ll}
$h\left( X,\ PY\right) =$ & $-h\left( X,\ \ \widehat{P}Y\right) =-h\left(
X,\ \frac{1}{\sqrt{5}}\left( 2\widehat{G}-I\right) Y\right) $ \\ 
& $-\left[ -\frac{1}{\sqrt{5}}h\left( X,\ Y\right) +\frac{2}{\sqrt{5}}%
h\left( \widehat{G}X,\ Y\right) \right] =-\left[ h\left( \frac{1}{\sqrt{5}}%
\left( 2\widehat{G}-I\right) X,\ Y\right) \right] $ \\ 
& $=-h\left( \widehat{P}X,\ Y\right) =h\left( PX,\ Y\right) $%
\end{tabular}%
\end{equation*}

$\blacksquare $

We immediately get, from Proposition (2.1), the following

\begin{proposition}
(2.2): \ \ \ Let \textit{an almost product }structure \ \ $P$ \ \ and an
almost \textit{golden} structure \ \ $G$ \ \ form a twin pair \ $\left\{ P,\
G\right\} $ \ on a smooth manifold \ \ $M.$
\end{proposition}

$\left( A\right) $\textit{: \ The following statements are equivalent:}

$\ i)$ \textit{\ }$\left( M,\ h,\ P\right) $\textit{\ \ is an almost product
Riemannian manifold.}

$ii)$\textit{\ \ }$\left( M,\ h,\ \ \widehat{P}\right) $\textit{\ \ is an
almost product Riemannian manifold.}

$iii)$\textit{\ \ }$\left( M,\ h,\ G\right) $\textit{\ \ is an almost
golden-Riemannian manifold.}

$iv)$\textit{\ \ }$\left( M,\ h,\ \ \widehat{G}\right) $\textit{\ \ is an
almost golden-Riemannian manifold.}

$\left( B\right) $\textit{: \ The following statements are equivalent}

$\ i)$\textit{\ \ }$\left( M,\ h,\ P\right) $\textit{\ \ is an almost
para-Hermitian manifold.}

$\ ii)$\textit{\ \ }$\left( M,\ h,\ \ \widehat{P}\right) $\textit{\ \ is an
almost para-Hermitian manifold.}

$iii)$\textit{\ \ }$\left( M,\ h,\ G\right) $\textit{\ \ is an almost
golden-Hermitian manifold.}

$iv)$\textit{\ \ }$\left( M,\ h,\ \ \widehat{G}\right) $\textit{\ \ is an
almost golden-Hermitian manifold.}

\begin{definition}
(2.2): \ An almost product manifold \ $\left( M,\ P\right) $ \ and an almost
golden manifold \ $\left( M,\ G\right) $ \ are said to be twins if \ $P$ \
and \ $G$ \ are twins (on the same manifold \ $M$).
\end{definition}

\begin{remark}
(2.1): \ It is obvious that \ $\left( M,\ P\right) $ \ and \ $\left( M,\
G\right) $ \ are twins \ if and only if \ $\left( M,\ \ \widehat{P}\right) $
\ and \ $\left( M,\ \widehat{G}\right) $ \ are twins. \ \ \ \ \ \ \ \ \ \ \
\ \ \ \ \ \ \ \ \ \ \ \ \ \ \ \ \ \ \ \ \ \ \ \ \ \ \ \ \ \ $\blacksquare $
\end{remark}

For an almost product (or golden) manifold \ \textit{\ }$\left( M,\ \varphi
\right) ,$ \ $\varphi $ \ is said to be integrable if its Nijenhuis tensor
field \ $\mathcal{N}_{\varphi }$ \ vanishes, $\left( \left[ \mathbf{3,\ 9}%
\right] \ \right) $. That is, \ $\forall \ X,\ Y\in \Gamma \left( TM\right) $%
\begin{equation*}
\mathcal{N}_{\varphi }\left( X,\ Y\right) =\varphi ^{2}\left[ X,\ Y\right] +%
\left[ \varphi X,\ \varphi Y\right] -\varphi \left[ \varphi X,\ Y\right]
-\varphi \left[ X,\ \varphi Y\right] =0.
\end{equation*}

For an almost product (or golden) manifold \ \textit{\ }$\left( M,\ \varphi
\right) $ \ with integrable \ $\varphi $ \ we drop the adjective " \textit{%
almost} " and then simply call it \textit{product (}or\textit{\ golden)
manifold.}

\begin{lemma}
(2.3): \ $\left[ \mathbf{3}\right] ,$ For a twin pair \ $\left\{ P,\
G\right\} $ \ on a manifold \ $M$ \ with any linear connection \ $\widetilde{%
\nabla }$ one has \ 
\begin{equation*}
5\mathcal{N}_{P}=4\mathcal{N}_{G}\ \ \ \ \ \ \ \ \ \ \ \ \ \ \text{and \ \ \
\ \ }\sqrt{5}\widetilde{\nabla }P=2\widetilde{\nabla }G.
\end{equation*}
\end{lemma}

This lemma gives immediately:

\begin{corollary}
(2.1): \ \ Let $\ \left\{ P,\ G\right\} $ \ be a twin pair\ on a manifold \ $%
M,$ \ then we have:
\end{corollary}

$P$\textit{\ \ is integrable if and only if \ }\ $\widehat{P}$\textit{\ \ is
integrable if and only if }\ $\widehat{G}$\textit{\ \ is integrable if and
only if \ \ }$G$\textit{\ \ is integrable.}

\begin{lemma}
(2.4): \ $\left( \left[ \mathbf{9}\right] ,\ Pg:150-151\right) $ \ For an
almost product manifold \textit{\ }$\left( M,\ P\right) ,$
\end{lemma}

$\ \ i)$ \ There always exist a linear connection \ $\bar{\nabla}\ \ $on \ $%
M $ \ with \ \ $\bar{\nabla}P=0.$

({\small Note that despite that\ \ }$\bar{\nabla}P=0$, {\small the} {\small %
almost product structure }$\ {\small P}${\small \ \ may not to be integrable
unless }$\bar{\nabla}${\small \ \ is symmetric.})

$\ ii)$ \ For any symmetric linear connection $\ \check{\nabla}$ \ on \ $M$
\ 
\begin{equation*}
\mathcal{N}_{P}\left( X,\ Y\right) =\left( \check{\nabla}_{PX}P\right)
Y-\left( \check{\nabla}_{PY}P\right) X-P\left( \left( \check{\nabla}%
_{X}P\right) Y\right) +P\left( \left( \check{\nabla}_{Y}P\right) X\right)
\end{equation*}%
and therefore,\ If \ $\check{\nabla}P=0$ \ then \ $P$ \ is integrable.

$iii)$ \ If \ $P$ \ is integrable then there always exist a symmetric linear
connection \ $\overset{s}{\nabla }$ \ on \ $M$ \ with \ $\overset{s}{\nabla }%
P=0.$

$\blacksquare $

From Corollary $\left( 2.1\right) $ \ and Lemma $\left( 2.4\right) $ \ one
gets:

\begin{corollary}
(2.2): \ Let \ $\left\{ P,\ G\right\} $ \ be a twin pair of almost product
and almost golden structures on a smooth manifold \ $\left( M,\ h\right) $\
\ with a nondegenerate metric \ $h.$ \ Then, for the Levi-Civita connection
\ $\nabla $ \ on \ $\left( M,\ h\right) $\ \ one has:
\end{corollary}

$A)$ \ The following are equivalent: \ 

$\ i)\ \ \nabla P=0.\ \ \ \ \ \ $

$ii)\ \ \nabla \ \widehat{G}=0.\ \ \ \ \ \ \ $

$iii)\ \ \nabla G=0.\ \ \ \ \ \ \ $

$iv)\ \ \nabla \ \widehat{P}=0.\ \ $

$B)$\ \ If \ $\nabla P=0$ \ \ then \ \ $P$, \ \ $\widehat{P}$, \ \ $\widehat{%
G}$ \ and \ $G$ \ \ are all integrable.\ \ \ \ \ \ \ \ \ \ \ \ \ \ \ \ \ 

\begin{remark}
(2.2): \ Note that
\end{remark}

$i)$ \ \ the above Corollary is true regardless of whether \ $h$ \ is pure
or hyperbolic or neither with respect to \ $P$ \ (and therefore with respect
to\ \ $\widehat{P},$ \ $G$ \ and \ \ $\widehat{G}$).

$ii)$ \ \ Integrability of \ $\varphi \left( =P,\ G\right) $ \ does not
imply that \ $\varphi $ \ is paralell\ (with respect to the metric ({\small %
Levi-Civita}) connection).\ \ \ \ \ \ \ \ \ \ \ \ \ \ \ \ \ \ \ \ \ \ \ \ \
\ \ \ \ \ \ \ \ \ \ \ \ \ \ \ \ \ \ \ \ \ \ $\blacksquare $

Let \ \ $\left( M,\ h,\ \varphi \left( =P,\ G\right) \right) $ \ be an
almost product (or golden) manifold with a metric \ $h$ \ which is pure or
hyperbolic with respect to \ $\varphi .$ \ Then due to the above lemma \ $%
(2.4)/\left( iii\right) ,$ \ an integrable $\varphi $ is always parallel
with respect to some symmetric connection $\ \overset{s}{\nabla }$ \ anyway.
However $\overset{s}{\nabla }h=0$ \ need not be true, that is, $\overset{s}{%
\nabla }$\ \ need not be the Levi-Civita connection. \ The question here is
that what extra condition should be imposed so that integrability of \ $%
\varphi ,$ \ together with the imposed condition,\ guarantees that \ $%
\varphi $ \ is paralell under the Levi-Civita connection? Answer to this
question will differ depending on whether the metric \ $h$ \ is pure or
hyperbolic with respect to \ $\varphi .$

From here on, unless otherwise stated, the connections involved will be the
Levi-Civita ones and denoted by \ $\nabla $.

$\mathbf{I}:$ \ \textbf{The hyperbolic case:} Even though this case\ is well
known for \ $\varphi =P$, \ $\left( see\ \ \left[ \mathbf{2,\ 11}\right]
\right) $, \ we will give an outline to some extend.

Let \ \ $\left( M,\ h,\ \varphi \left( =P,\ G\right) \right) $ \ be an
almost para-Hermitian manifold with its Levi-Civita connection $\ \nabla .$
Set \ 
\begin{equation*}
\Omega _{P}\left( X,\ Y\right) =\Omega \left( X,\ Y\right) =h\left( PX,\
Y\right) ;\ \ \ \ \ \ \ \ \ \ \ \ \ \ \ \forall X,\ Y\in \Gamma \left(
TM\right) .
\end{equation*}%
$\Omega _{P}$\ \ is a ( $P-$associated ) 2-form on \ $M$ \ and it is called
\ "\textit{fundamental 2-form}" or \ "\textit{para-Kaehler form}". The
exterior differential \ $d\Omega $ \ is a 3-form on \ $M$ \ given by, $\
\left( \left[ \mathbf{6}\right] \right) ,$ 
\begin{equation*}
\begin{tabular}{ll}
$d\Omega \left( X,\ Y,\ Z\right) =$ & $\nabla _{X}\left( \Omega \left( Y,\
Z\right) \right) -\nabla _{Y}\left( \Omega \left( X,\ Z\right) \right)
+\nabla _{Z}\left( \Omega \left( X,\ Y\right) \right) $ \\ 
& $-\Omega \left( \left[ X,\ Y\right] ,\ \ Z\right) -\Omega \left( \left[
Y,\ Z\right] ,\ \ X\right) +\Omega \left( \left[ X,\ Z\right] ,\ \ Y\right) $%
\end{tabular}%
\end{equation*}

which can also be expressed as 
\begin{equation}
d\Omega \left( X,\ Y,\ Z\right) =\left( \nabla _{X}\Omega \right) \left( Y,\
Z\right) -\left( \nabla _{Y}\Omega \right) \left( X,\ Z\right) +\left(
\nabla _{Z}\Omega \right) \left( X,\ Y\right) .  \tag{$\left( 2.1\right) $}
\end{equation}

\begin{definition}
(2.3):
\end{definition}

$A:$ \ $\left( \left[ \mathbf{11}\right] \right) $

$i)$ \ \ An almost para-Hermitian manifold \ $\left( M,\ h,\ P\right) $ \ \
is called \textit{almost para-Kaehler} \ if\ its para-Kaehler form \ $\Omega
_{P}$ \ is closed, $i.e$. $d\Omega _{P}=0.$ \ 

$ii)$ \ An almost para-Kaehler manifold \ $\left( M,\ h,\ P\right) $ \ with
integrable $\ P$ $\ $is called \textit{para-Kaehler }$.$ \ 

$B:$

$i)$ \ \ An almost golden-Hermitian manifold $\ \left( M,\ h,\ G\right) $ \
is called \textit{almost golden-Kaehler }if\textit{\ }the para-Kaehler form
\ $\Omega _{\mathcal{R}}$\ \ is closed, $i.e.$\ \ $d\Omega _{\mathcal{R}}=0,$
\ where $\ \mathcal{R=}P_{G}$ \ is the \ $G-$associated product structure
and $\ \Omega _{\mathcal{R}}$ $\ $is $\ $the $\ \mathcal{R}$-associated
2-form.

$ii)$ \ An almost golden-Kaehler manifold \ $\left( M,\ h,\ G\right) $ \
with integrable\ $G$ \ is called \textit{golden-Kaehler.}

\begin{proposition}
(2.3): \ Let \ $\left\{ P,\ G\right\} $ \ be twin structures on \ $\left(
M,h\right) $ \ \ with a hyperbolic metric \ $h$ \ with respect to \ $P$ \ (
and therefore with respect to \ $G$ ). \ Then the following statements are
equivalent:
\end{proposition}

\ $i)$ \ \ \textit{The manifold }$\ \left( M,\ h,\ P\right) $\textit{\ \ is
almost para-Kaehler, that is, \ }$d\Omega _{P}=0.$

$\ ii)$\textit{\ \ \ The manifold \ }$\left( M,\ h,\ G\right) $\textit{\ \
is an almost golden-Kaehler, that is, }$d\Omega _{\mathcal{R}}=0.$

$iii)$\textit{\ \ \ The manifold \ }$\left( M,\ h,\ \widehat{P}\right) $%
\textit{\ \ is an almost para-Kaehler, that is, \ }$d\Omega _{\widehat{P}}=0$

$iv)$\textit{\ \ \ The manifold \ }$\left( M,\ h,\ \widehat{G}\right) $%
\textit{\ \ is an almost golden-Kaehler, that is,} \ $d\Omega _{\widehat{%
\mathcal{R}}}=0.$ \ 

\begin{proof}
The result follows from the fact that \ $\mathcal{R=}P_{G}=P$ \ since \ $P$ $%
\ $and $\ G$ \ are twins.
\end{proof}

\begin{lemma}
(2.5): $\left( \left[ \mathbf{5}\right] \right) $ \ Let \ \ $\left( M,\ h,\
P\right) $ \ be an almost para-Hermitian manifold with its Levi-Civita
connection $\ \nabla $ \ and para-Kaehler form \ $\Omega .$ \ Then the
following relation holds:$\ \ \ \ \forall X,\ Y,\ Z\in \Gamma \left(
TM\right) $ 
\begin{equation*}
2h\left( \left( \nabla _{X}P\right) Y,\ Z\right) +3d\Omega \left( X,\ Y,\
Z\right) +3d\Omega \left( X,\ PY,\ PZ\right) +h\left( \mathcal{N}_{P}\left(
Y,\ Z\right) ,\ PX\right) =0.
\end{equation*}
\end{lemma}

\begin{proposition}
(2.4): \ Let \ \ $\left( M,\ h,\ \varphi \left( =P,\ G\right) \right) $ \ be
an almost para-Hermitian or an almost golden-Hermitian manifold.
\end{proposition}

$A:$ \ \textit{Then the following are equivalent:}

\textit{\ }$i)$\textit{\ \ \ }$P$\textit{\ }$\ $\textit{is parallel with
respect to the Levi-Civita connection }$\ \nabla $\textit{, that is }$\
\nabla P=0.$\textit{\ \ }

$ii)$\textit{\ \ \ }$M$ \ is \textit{\ para-Kaehler (that is, \ }$\mathcal{N}%
_{P}=0$ $\ $and $\ d\Omega _{P}=0).$

$B:$\textit{\ \ Then the following are equivalent:}

\textit{\ }$i)$\textit{\ \ \ }$G$\textit{\ }$\ $\textit{is parallel with
respect to the Levi-Civita connection }$\ \nabla $\textit{, that is, }$\
\nabla G=0.$\textit{\ \ }

$ii)$\textit{\ \ }$M$ \ is \textit{\ golden-Kaehler (that is \ }$\mathcal{N}%
_{G}=0$ $\ $and $\ d\Omega _{\mathcal{R}}=0),$\textit{\ \ where }$\ \mathcal{%
R}=P_{G}$\textit{\ \ is the \ }$G$-\textit{associated product structure.}

\begin{proof}
: $\ A$ \ $\left( \text{ \ }\left[ \mathbf{11}\right] \right) :$
\end{proof}

$\left( i\right) \Rightarrow \left( ii\right) :$ \ Note that 
\begin{equation}
\nabla \left( \Omega \left( X,\ Y\right) \right) =\left( \nabla \Omega
\right) \left( X,\ Y\right) +\Omega \left( \nabla X,\ Y\right) +\Omega
\left( X,\ \nabla Y\right)  \tag{$\left( 2.2\right) $}
\end{equation}%
On the other hand, since \ $\nabla P=0$ \ and \ $\nabla h=0,$ \ we have

\begin{equation}
\begin{tabular}{ll}
$\nabla \left( \Omega \left( X,\ Y\right) \right) =\nabla \left( h\left(
PX,\ Y\right) \right) $ & $=\left( \nabla h\right) \left( PX,\ Y\right)
+h\left( \nabla \left( PX\right) ,\ Y\right) +h\left( PX,\ \nabla Y\right) $
\\ 
& $=h\left( P\left( \nabla X\right) ,\ Y\right) +h\left( PX,\ \nabla
Y\right) $ \\ 
& $=\Omega \left( \nabla X,\ Y\right) +\Omega \left( X,\ \nabla Y\right) $%
\end{tabular}
\tag{$\left( 2.3\right) $}
\end{equation}

But then the equalities $\ \left( 2.2\right) $ \ and \ $\left( 2.3\right) $
\ give us that \ $\nabla \Omega =0.\ $\ So, from \ $\left( 2.1\right) ,$ \
we get \ $d\Omega _{P}=0.$ \ The equality \ $\mathcal{N}_{P}=0$ \ follows
from Lemma $\left( 2.5\right) .$

$\left( ii\right) \Rightarrow \left( i\right) :$ \ This follows directly
from Lemma $\left( 2.5\right) .$

$B:$ \ Since $\ \left\{ \mathcal{R},\ G\right\} $ \ is a twin pair on \ $M,$
\ the reqired equivalence follows from part \ $\left( A\right) $. \ 

\begin{remark}
(2.3): \ Let \ $J$ \ be a \ $\left( 1,\ 1\right) $-tensor field with \ \ $%
J^{2}=-I$ \ on a Riemannian manifold \ \ $\left( M,\ g\right) ,$ \ where \ $%
g $ \ is hyperbolic with respect to \ $J\ \ \ $and \ $\widehat{J}=-J$ \ is
the conjugate \ of $\ J.$ \ Then \ $J$ \ and \ $\left( M,\ g,\ J\right) $ \
are called almost complex structure and almost Hermitian manifold
respectively. \ In this case it is well known that Proposition $\ (2.4)/A$ \
is also valid when \ $P$ \ is replaced by \ $J.$ \ \ \ \ \ \ \ \ \ \ \ \ \ \
\ \ \ \ \ \ \ \ \ \ \ \ \ \ \ \ \ \ \ \ $\blacksquare $
\end{remark}

$\mathbf{II}:$ \ \textbf{The pure case}:

Let \ \ $\left( M,\ h,\ \varphi \left( =P,\ G\right) \right) $ \ be an
almost product (or an almost golden) Riemannian manifold (so that \ $h$ \ is
pure) with its Levi-Civita connection $\ \nabla .$ The so called "Tachibana
operator"%
\begin{equation*}
\phi _{\varphi }:\mathcal{\Im }_{2}^{0}\left( M\right) \rightarrow \mathcal{%
\Im }_{3}^{0}\left( M\right)
\end{equation*}%
from the set of all \ $\left( 0,2\right) -$tensor fields into the set of all
\ $\left( 0,3\right) -$tensor fields over \ $M$ \ is given by , $\left( see%
\text{ \ }\left[ \mathbf{8,\ 12}\right] \right) :\forall \ u\in \mathcal{\Im 
}_{2}^{0}\left( M\right) $ \ \ and \ \ $\forall \ X,\ Y,\ Z\in \Gamma \left(
TM\right) $

\begin{equation*}
\begin{tabular}{ll}
$\left( \phi _{\varphi }u\right) \left( X,\ Y,\ Z\right) =$ & $\left(
\varphi X\right) \left( u\left( Y,\ Z\right) \right) -X\left( u\left(
\varphi Y,\ Z\right) \right) $ \\ 
& $+u\left( \left( \mathcal{L}_{_{Y}}\varphi \right) X,\ Z\right) +u\left(
Y,\ \left( \mathcal{L}_{_{Z}}\varphi \right) X\right) $%
\end{tabular}%
\end{equation*}%
where \ $\mathcal{L}\varphi $ \ is the Lie derivative of \ $\varphi .$

In particular, for the pure metric \ $h$ \ with respect to \ $\varphi $, the
above equality takes the form \ $\left( see\text{ \ \ }\left[ \mathbf{8,\ 12}%
\right] \right) :\ \ \ \forall \ X,\ Y,\ Z\in \Gamma \left( TM\right) $%
\begin{equation*}
\left( \phi _{\varphi }h\right) \left( X,\ Y,\ Z\right) =-h\left( \left(
\nabla _{_{X}}\varphi \right) Y,\ Z\right) +h\left( \left( \nabla
_{_{Y}}\varphi \right) X,\ Z\right) +h\left( \left( \nabla _{_{Z}}\varphi
\right) X,\ Y\right) .
\end{equation*}

Now let us define another operator%
\begin{equation*}
\mathbf{\Psi }_{\varphi }:\mathcal{\Im }_{2}^{0}\left( M\right) \rightarrow 
\mathcal{\Im }_{3}^{0}\left( M\right)
\end{equation*}%
by, $\ \forall \ X,\ Y,\ Z\in \Gamma \left( TM\right) $ \ 
\begin{equation*}
\left( \mathbf{\Psi }_{\varphi }u\right) \left( X,\ Y,\ Z\right) =\left(
\phi _{\varphi }u\right) \left( X,\ Y,\ Z\right) +\left( \phi _{\varphi
}u\right) \left( Z,\ Y,\ X\right)
\end{equation*}

\begin{lemma}
(2.6): $\left( \text{\ }\left[ \mathbf{8,\ 12}\right] \right) $ \ Let \ \ $%
\left( M,\ h,\ \varphi \left( =P,\ G\right) \right) $ \ be an almost product
(or an almost golden) Riemannian manifold with its Levi-Civita connection $\
\nabla .$ Then
\end{lemma}

\ $i)$ \ $\forall \ X,\ Y,\ Z\in \Gamma \left( TM\right) $%
\begin{equation*}
\left( \mathbf{\Psi }_{\varphi }h\right) \left( X,\ Y,\ Z\right) =2h\left(
\left( \nabla _{_{Y}}\varphi \right) X,\ Z\right)
\end{equation*}

\ $ii)$ \ The following are equivalent:

$a^{\circ })$ \ $\left( \mathbf{\Psi }_{\varphi }h\right) =0$

$b^{\circ })$ \ $\left( \phi _{\varphi }h\right) =0$

$c^{\circ })$ \ $\nabla \varphi =0$

$\ \ \ \ \ \ \ \ \ \ \ \ \ \ \ \ \ \ \ \ \ \ \ \ \ \ \ \ \ \ \ \ \ \ \ \ \ \
\ \ \ \ \ \ \ \ \ \ \ \ \ \ \ \ \ \ \ \ \ \ \ \ \ \ \ \ \ \ \ \ \ \ \ \ \ \
\ \ \ \ \ \ \ \ \ \ \ \ \ \ \ \ \ \ \ \ \ \ \ \ \ \ \ \ \ \ \ \ \ \ \ \ \ \
\ \ \ \ \ \ \ \ \ \ \ \ \ \ \ \ \ \ \ \ \ \ \ \ \ \blacksquare $

Now set a condition on \ $\mathbf{\Psi }_{\varphi }:$%
\begin{equation}
\left( \mathbf{\Psi }_{\varphi }h\right) \left( X,\text{ }Y,\ Z\right)
=\left( \mathbf{\Psi }_{\varphi }h\right) \left( Y,\ X,\ Z\right) +\left( 
\mathbf{\Psi }_{\varphi }h\right) \left( \varphi Y,\ \varphi X,\ Z\right) ,\
\ \ \forall \ X,\ Y,\ Z\in \Gamma \left( TM\right) 
\tag{$\varphi
\left(
\ast \right) $}
\end{equation}

\begin{proposition}
(2.5): \ Let \ \ $\left( M,\ h,\ \varphi \left( =P,\ G\right) \right) $ \ be
an almost product (or an almost golden) Riemannian manifold Then
\end{proposition}

$A:$ \ The following are equivalent:

\ $i)$ \ \ $P$ $\ $is parallel with respect to the Levi-Civita connection $\
\nabla $, that is, $\ \nabla P=0.$ \ 

$ii)$ \ $P$ $\ $is integrable,$\ $that is, \ $\mathcal{N}_{P}=0$ \ \ and$\ $%
the condition \ $P\left( \ast \right) $ \ holds.

$B:$ \ The the following are equivalent:

\ $i)$ \ $G$ $\ $is parallel with respect to the Levi-Civita connection $\
\nabla $, that is, $\ \nabla G=0.$ \ 

$ii)$ \ $G$ $\ $is integrable,$\ $that is, \ $\mathcal{N}_{G}=0$ \ \ and \ \
the condition $\ $\ $\mathcal{R}\left( \ast \right) $ \ holds.

Here \ \ $\mathcal{R=}P_{G}$ \ is the twin product structure of \ $G,$ $\ $so%
$\ $that,\ \ $\left\{ G,\ \mathcal{R}\right\} $ \ form a twin pair. \ 

\begin{proof}
: \ 
\end{proof}

$A:$

$\left( i\right) \Rightarrow \left( ii\right) :$ \ Since \ $\nabla P=0$ \ by
the assumption, we have from Lemma\ $\left( 2.4\right) /\left( ii\right) $ \
that $\ \mathcal{N}_{P}=0,$ \ and from Lemma $\left( 2.6\right) $ \ the
condition $P\left( \ast \right) $\ \ follows.

$\left( ii\right) \Rightarrow \left( i\right) :$ \ From Lemma $\left(
2.4\right) /\left( ii\right) $ \ we have: \ $\forall \ X,\ Y,\ Z\in \Gamma
\left( TM\right) $%
\begin{equation*}
\begin{tabular}{ll}
$\mathcal{N}_{P}\left( X,\ Y\right) $ & $=\left( \nabla _{PX}P\right)
Y-\left( \nabla _{PY}P\right) X-P\left( \left( \nabla _{X}P\right) Y\right)
+P\left( \left( \nabla _{Y}P\right) X\right) $ \\ 
& $=\left( \nabla _{PX}P\right) Y-\left( \nabla _{PY}P\right) X+\left(
\nabla _{X}P\right) \left( PY\right) -\left( \nabla _{Y}P\right) \left(
PX\right) $%
\end{tabular}%
\end{equation*}%
So%
\begin{equation*}
\begin{tabular}{ll}
$h\left( \mathcal{N}_{P}\left( X,\ Y\right) ,\ Z\right) $ & $=h\left( \left(
\nabla _{PX}P\right) Y,\ Z\right) -h\left( \left( \nabla _{PY}P\right) X,\
Z\right) $ \\ 
& $+h\left( \left( \nabla _{X}P\right) \left( PY\right) ,\ Z\right) -h\left(
\left( \nabla _{Y}P\right) \left( PX\right) ,\ Z\right) .$%
\end{tabular}%
\end{equation*}%
That is,%
\begin{equation*}
\begin{tabular}{ll}
$h\left( \mathcal{N}_{P}\left( X,\ Y\right) ,\ Z\right) -h\left( \left(
\nabla _{X}P\right) \left( PY\right) ,\ Z\right) $ & $=h\left( \left( \nabla
_{PX}P\right) Y,\ Z\right) $ \\ 
& $-h\left( \left( \nabla _{PY}P\right) X,\ Z\right) -h\left( \left( \nabla
_{Y}P\right) \left( PX\right) ,\ Z\right) .$%
\end{tabular}%
\end{equation*}%
Then using Lemma $\left( 2.6\right) ,$ we get; \ $\forall \ X,\ Y,\ Z\in
\Gamma \left( TM\right) $%
\begin{equation*}
\begin{tabular}{ll}
$2\left\{ h\left( \mathcal{N}_{P}\left( X,\ Y\right) ,\ Z\right) -h\left(
\left( \nabla _{X}P\right) \left( PY\right) ,\ Z\right) \right\} $ & $%
=\left( \mathbf{\Psi }_{P}h\right) \left( Y,\ PX,\ Z\right) $ \\ 
& $-\left( \mathbf{\Psi }_{P}h\right) \left( X,\ PY,\ Z\right) -\left( 
\mathbf{\Psi }_{P}h\right) \left( PX,\ Y,\ Z\right) .$%
\end{tabular}%
\end{equation*}%
Exchanging \ $\ X,$ $\ $with$\ \ \ Y,$ \ this equation reads:%
\begin{equation*}
\begin{tabular}{ll}
$2\left\{ h\left( \mathcal{N}_{P}\left( Y,\ X\right) ,\ Z\right) -h\left(
\left( \nabla _{Y}P\right) \left( PX\right) ,\ Z\right) \right\} $ & $%
=\left( \mathbf{\Psi }_{P}h\right) \left( X,\ PY,\ Z\right) $ \\ 
& $-\left( \mathbf{\Psi }_{P}h\right) \left( Y,\ PX,\ Z\right) -\left( 
\mathbf{\Psi }_{P}h\right) \left( PY,\ X,\ Z\right) .$%
\end{tabular}%
\end{equation*}%
Then putting \ $Y$ \ for \ $PY$ \ in the last equation ( doing this does not
alter the equation since \ $P$ \ is an isomorphism) we get; $\ \ \forall \
X,\ Y,\ Z\in \Gamma \left( TM\right) $

\begin{equation*}
\begin{tabular}{ll}
$2\left\{ h\left( \mathcal{N}_{P}\left( PY,\ X\right) ,\ Z\right) -h\left(
\left( \nabla _{PY}P\right) \left( PX\right) ,\ Z\right) \right\} $ & $%
=\left( \mathbf{\Psi }_{P}h\right) \left( X,\ Y,\ Z\right) $ \\ 
& $-\left( \mathbf{\Psi }_{P}h\right) \left( PY,\ PX,\ Z\right) -\left( 
\mathbf{\Psi }_{P}h\right) \left( Y,\ X,\ Z\right) .$%
\end{tabular}%
\end{equation*}%
But then under the assumptions that \ $\mathcal{N}_{P}=0$ \ and \ the
condition \ $P\left( \ast \right) $\ \ holds, the last equation gives us
that $\ $%
\begin{equation*}
h\left( \left( \nabla _{PY}P\right) \left( PX\right) ,\ Z\right) =0,\ \
\forall \ X,\ Y,\ Z\in \Gamma \left( TM\right)
\end{equation*}%
which means that \ $\nabla P=0.$

$B:$

$\left( i\right) \Rightarrow \left( ii\right) :$ \ By the assumption, \ $%
\nabla G=0$ \ and therefore \ $\nabla \mathcal{R}=0.$ \ So by part \ $\left(
A\right) $ \ above,\ we get \ $\mathcal{N}_{\mathcal{R}}=0,$ \ and therefore
\ $\mathcal{N}_{G}=0$ \ by Lemma $\left( 2.3\right) .$\ Also by part \ $%
\left( A\right) $, \ we get that the condition $\ \mathcal{R}\left( \ast
\right) $ \ holds.

$\left( ii\right) \Rightarrow \left( i\right) :$ \ By the assumption, \ $%
\mathcal{N}_{G}=0$ \ and therefore \ $\mathcal{N}_{\mathcal{R}}=0$ \ and the
condition \ $\mathcal{R}\left( \ast \right) $\ \ holds. So by part \ $\left(
A\right) ,$\ we get \ $\nabla \mathcal{R}=0,$ \ and therefore \ $\nabla G=0$
\ by Corollary $\left( 2.2\right) .$ \ \ $\blacksquare $

For an almost product (or an almost golden ) manifold \ $\left( M,\ h,\
\varphi \right) $ \ with a pure or hyperbolic metric \ $h$ \ with respect to
\ $\varphi ,$ \ and with its Levi-Civita connection $\nabla ,$ \ the
divergence \ $\func{div}\varphi \ \ $of \ $\varphi $ \ is given by, $\left[ 
\mathbf{6}\right] ,$%
\begin{equation*}
\func{div}\varphi =\sum_{i=1}^{m}h_{ii}\left( \nabla _{e_{i}}\varphi \right)
e_{i}.
\end{equation*}%
Here \ $\left\{ e_{_{1}},...,e_{_{m}}\right\} $ \ is a local orthonormal
frame field for \ $\Gamma \left( TM\right) $ \ and \ \ $h_{ii}=h\left(
e_{i},\ e_{i}\right) .$

\begin{definition}
(2.4): \ 
\end{definition}

$\left( A\right) :$ \ An almost product Riemannian manifold \ $\left( M,\
h,\ P\right) $ \ with its Levi-Civita connection $\ \nabla ,$\ is called

\ $i)$ \ \textit{locally product Riemannian manifold} \ if \ $P$ \ is
integrable, $\left[ \mathbf{8}\right] $.

$ii)$ \ \textit{almost decomposable product Riemannian manifold }\ if \ $%
P\left( \ast \right) $ \ holds.

$iii)$ \ \textit{locally decomposable product Riemannian manifold }\ if both
\ $P$ \ is integrable and \ $P\left( \ast \right) $ \ holds (that is, $\ P$
\ is parallel), $\left[ \mathbf{8}\right] $.

In particular, if \ $\left( M,\ h,\ P\right) $ \ is a $\ \mathbf{B}$%
-manifold \ ($resp:$ almost \ $\mathbf{B}$-manifold\ ) holding the condition 
$\ \ P\left( \ast \right) $ \ then it is also called \textit{%
para-holomorphic }$\ \mathbf{B}$\textit{-manifold, }$\left[ \mathbf{12}%
\right] ,$ \ ($resp:$ \textit{almost para-holomorphic\ }$B$\textit{-manifold}%
\ ).\ Note here that by the virtue of Proposition $\left( 2.5\right) ,$ \ if
\ $\left( M,\ h,\ P\right) $ \ is a para-holomorphic$\ \mathbf{B}$-manifold
\ then \ $\nabla P=0$, \ i.e. \ $P$ is parallel.

$\ iv)$ \ \textit{Semi decomposable product Riemannian manifold }if $\ \func{%
div}P=0$

$\left( B\right) :$ \ An almost golden Riemannian manifold \ $\left( M,\ h,\
G\right) $ \ with its Levi-Civita connection $\ \nabla ,$\ is called

\ \ $i)$ \ \textit{locally golden Riemannian manifold} \ if \ $G$ \ is
integrable, $\left[ \mathbf{8}\right] $.

$ii)$ \ \textit{almost decomposable golden Riemannian manifold }\ if \ $%
G\left( \ast \right) $ \ holds.

$iii)$ \ \textit{locally decomposable golden Riemannian manifold }\ if both
\ $G$ \ is integrable and \ $G\left( \ast \right) $ \ holds (that is, $\ G$
\ is parallel), $\left[ \mathbf{8}\right] $.

$iv)$ \ \textit{Semi decomposable golden Riemannian manifold} if $\ \func{div%
}G=0$\ \ \ \ \ \ \ \ \ \ \ \ \ \ \ \ \ \ \ \ \ \ \ \ \ \ \ \ \ $\blacksquare 
$

Define a bilinear map, $\left[ \mathbf{7}\right] ,$ 
\begin{equation*}
S_{\varphi }:\Gamma \left( TM\right) \times \Gamma \left( TM\right)
\rightarrow \Gamma \left( TM\right)
\end{equation*}%
on a manifold \ $\left( M,\ h,\ \varphi \left( =P,\ G\right) \right) $ \
with the Levi-Civita connection $\ \nabla $ \ \ by%
\begin{equation*}
S_{\varphi }\left( X,\ Y\right) =\left( \nabla _{_{X}}\varphi \right)
Y+\varphi \left( \nabla _{_{\varphi X}}\varphi \right) Y\text{ \ \ \ \ \ \ }%
\forall \ X,\ Y\in \Gamma \left( TM\right) .
\end{equation*}

\begin{lemma}
(2.7):
\end{lemma}

$\left( A\right) :$ \ For $\varphi =P$ $\ $and $\ \forall \ X,\ Y\in \Gamma
\left( TM\right) $ \ we have

\ $i)$ \ 
\begin{equation*}
S_{P}\left( X,\ Y\right) =\left( \nabla _{_{X}}P\right) Y-\left( \nabla
_{_{PX}}P\right) \left( PY\right) .
\end{equation*}

$ii)$ \ 
\begin{equation*}
P\left( S_{P}\left( X,\ Y\right) \right) =-S_{P}\left( X,\ PY\right)
=S_{P}\left( PX,\ Y\right) .
\end{equation*}

$iii)$

\begin{equation*}
S_{P}\left( X,\ Y\right) =0,\ \ \forall \ X,\ Y\in \Gamma \left( \mathcal{E}%
_{\left( 1\right) }\right) \ \ \text{and \ }S_{P}\left( X,\ Y\right) =0,\
\forall \ X,\ Y\in \Gamma \left( \mathcal{E}_{\left( -1\right) }\right) .
\end{equation*}

$\left( B\right) :$ \ For $\varphi =G$ $\ $and $\ \forall \ X,\ Y\in \Gamma
\left( TM\right) $ \ we have

\ $i)$ 
\begin{equation*}
S_{G}\left( X,\ Y\right) =\left( \nabla _{_{X}}G\right) Y-\left( \nabla
_{_{GX}}G\right) \left( GY\right) +\left( \nabla _{_{GX}}G\right) Y.
\end{equation*}

$ii)$%
\begin{equation*}
S_{G}\left( X,\ Y\right) =0,\text{ \ }\forall \ X,\ Y\in \Gamma \left( 
\mathcal{E}_{\left( \sigma \right) }\right) \ \ \ \text{and}\ \ \
S_{G}\left( X,\ Y\right) =0,\text{ \ }\ \forall \ X,\ Y\in \Gamma \left( 
\mathcal{E}_{\left( \bar{\sigma}\right) }\right) .\text{\ }
\end{equation*}

\begin{proof}
:
\end{proof}

Using the facts that \ \ $P\left( \left( \nabla P\right) X\right) =-\ \left(
\nabla P\right) \left( PX\right) $ \ \ and \ $G\left( \left( \nabla G\right)
Y\right) =-\left( \nabla G\right) \left( GY\right) +\left( \nabla G\right) Y$%
\ \ we get $\ A/\left( i\right) $ $\ \ $and \ \ $B/\left( i\right) .$ \ \
Next, \ \ $A/\left( ii\right) $ $\ \ $and \ \ $A/\left( iii\right) $ \ \ are
easy. \ For \ $B/\left( ii\right) $ \ let $X,\ Y\in \Gamma \left( \mathcal{E}%
_{\left( \sigma \right) }\right) $, then

\begin{equation*}
\begin{tabular}{ll}
$S_{G}\left( X,\ Y\right) $ & $=\left( \nabla _{_{X}}G\right) Y-\left(
\nabla _{_{GX}}G\right) \left( GY\right) +\left( \nabla _{_{GX}}G\right)
Y=\left( \nabla _{_{X}}G\right) Y-\sigma ^{2}\left( \nabla _{_{X}}G\right)
\left( Y\right) +\sigma \left( \nabla _{_{X}}G\right) Y$ \\ 
& $=\left( 1-\sigma ^{2}+\sigma \right) \left( \nabla _{_{X}}G\right) Y=0,$
\ \ since \ \ $\sigma ^{2}=1+\sigma .$%
\end{tabular}%
\end{equation*}%
By the same argument we also get that 
\begin{equation*}
S_{G}\left( X,\ Y\right) =0,\text{ \ }\ \ \ \forall \ X,\ Y\in \Gamma \left( 
\mathcal{E}_{\left( \bar{\sigma}\right) }\right) ,
\end{equation*}%
which completes the proof.

\begin{lemma}
(2.8):$\left( c.f\ \ \left[ \mathbf{7}\right] \right) $ \ On an almost
product Riemannian or an almost para Hermitian \ manifold $\ \ \left( M,\
h,\ P\right) ,$ \ \ the following statements are equivalent:
\end{lemma}

\ $i)$ \ $S_{P}\left( X,\ Y\right) =0\ \ \ \ \ \ \ \ \forall \ X,\ Y\in
\Gamma \left( TM\right) .$

\ $ii)$ \ $S_{P}\left( X,\ PX\right) =0\ \ \ \ \ \forall \ X\in \Gamma
\left( TM\right) .$

$iii)$ \ $S_{P}\left( X,\ X\right) =0\ \ \ \ \ \ \ \forall \ X\in \Gamma
\left( TM\right) .$

\begin{proposition}
(2.6): \ \ For an almost product Riemannian manifold \ $\ M_{P}=\left( M,\
h,\ P\right) $ and \ an almost golden Riemannian manifold \ $\ M_{G}=\left(
M,\ h,\ G\right) ,$ \ 
\end{proposition}

$\left( A\right) :$ \ on \ \ $M_{P}$

$\ \ i)\ \ \ \nabla P=0$ \ \ if and only if \ \ $S_{P}=0\ \ \ $\ if and only
if \ \ $\left( \mathbf{\Psi }_{P}h\right) =0\ \ $\ if and only if \ \ \ $%
\left( \phi _{P}h\right) =0.$

$ii)\ \ \ \nabla G=0$ \ \ \ if and only if $\ \ S_{G}=0\ \ \ $if and only if
\ \ $\left( \mathbf{\Psi }_{G}h\right) =0\ \ $\ if and only if \ \ \ $\left(
\phi _{G}h\right) =0$. \ \ 

$\left( B\right) :$ \ If \ $M_{P}$ \ and \ \ $M_{G}.$ \ are twin manifolds
then the following are equivalent:

$\ \ i)\ \ \ \nabla P=0$ \ \ on \ \ $M_{P}.$ \ \ 

$\ ii)\ \ S_{P}=0\ \ \ $on\ $\ \ M_{P}.$

$iii)\ \ \ \nabla G=0$ \ \ on \ \ $M_{G}.$

$iv)\ \ \ S_{G}=0\ \ \ $on $\ \ M_{G}.\ $

$\ v)\ \ \ \nabla \widehat{G}=0\ $\ on$\ \ M_{\widehat{G}}.$

$vi)$ $\ \ S_{\widehat{P}}=0\ \ $on$\ \ M_{\widehat{P}}.$

$vii)$ $\ \ \nabla \widehat{P}=0$ \ \ on \ \ $M_{\widehat{P}}.$

\begin{proof}
:
\end{proof}

$\left( A\right) :$ \ 

$\left( i\right) :$ $\ \ $If $\ \nabla P=0\ \ $\ then obviously \ \ $%
S_{P}=0. $

Conversely, assume that $\ \ S_{P}=0.$ \ Then for \ $X\in \Gamma \left( 
\mathcal{E}_{\left( 1\right) }\right) $ $\ \ $and $\ \ Y\in \Gamma \left( 
\mathcal{E}_{\left( -1\right) }\right) $ $\ \ $\ 
\begin{equation*}
S_{P}\left( Y,\ X\right) =\left( \nabla _{_{X}}P\right) Y-\left( \nabla
_{_{PX}}P\right) PY=2\left( \nabla _{_{X}}P\right) Y=0,
\end{equation*}%
which gives 
\begin{equation}
\left( \nabla _{X}P\right) Y=0;\ \ \ \ \ \ \ \forall \ X\in \Gamma \left( 
\mathcal{E}_{\left( 1\right) }\right) \ \ \text{and}\ \ \forall \ Y\in
\Gamma \left( \mathcal{E}_{\left( -1\right) }\right) . 
\tag{$\left(
2.4\right) $}
\end{equation}%
By a similar argument\ we get 
\begin{equation}
\left( \nabla _{Y}P\right) X=0;\ \ \ \ \ \ \ \forall \ X\in \Gamma \left( 
\mathcal{E}_{\left( 1\right) }\right) \ \ \text{and}\ \ \forall \ Y\in
\Gamma \left( \mathcal{E}_{\left( -1\right) }\right) . 
\tag{$\left(
2.5\right) $}
\end{equation}%
From \ $\left( 2.4\right) $ \ we get \ 
\begin{equation*}
\left( \nabla _{X}P\right) Y=\nabla _{X}\left( PY\right) -P\left( \nabla
_{X}Y\right) =-\nabla _{X}Y-P\left( \nabla _{X}Y\right) =0.
\end{equation*}%
So,\ 
\begin{equation}
\nabla _{X}Y\in \Gamma \left( \mathcal{E}_{\left( -1\right) }\right) \ \ \ \
\ \ \ \forall \ X\in \Gamma \left( \mathcal{E}_{\left( 1\right) }\right) \ \ 
\text{and}\ \ \forall \ Y\in \Gamma \left( \mathcal{E}_{\left( -1\right)
}\right) .  \tag{$\left( 2.6\right) $}
\end{equation}%
On the other hand, \ $\forall \ X,\ Z\in \Gamma \left( \mathcal{E}_{\left(
1\right) }\right) \ \ $and$\ \ \forall \ Y\in \Gamma \left( \mathcal{E}%
_{\left( -1\right) }\right) $%
\begin{equation*}
X\left( h\left( Y,Z\right) \right) =h\left( \nabla _{X}Y,\ Z\right) +h\left(
Y,\ \nabla _{X}Z\right) =0,\ \ \ \ \ \ \ \ \ \text{since \ \ }h\left(
Y,Z\right) =0.
\end{equation*}%
Using \ $\left( 2.6\right) ,$ \ this gives\ that \ \ $h\left( Y,\ \nabla
_{X}Z\right) =0,$ \ $\ \ \forall \ Y\in \Gamma \left( \mathcal{E}_{\left(
-1\right) }\right) $ \ \ and therefore $\ \nabla _{X}Z\in \Gamma \left( 
\mathcal{E}_{\left( 1\right) }\right) .$ \ But then,%
\begin{equation*}
P\left( \nabla _{X}Z\right) =\nabla _{X}Z=\nabla _{X}\left( PZ\right)
\end{equation*}%
which gives \ 
\begin{equation}
\left( \nabla _{X}P\right) Z=0;\ \ \ \ \ \ \ \forall \ X,\ Z\in \Gamma
\left( \mathcal{E}_{\left( 1\right) }\right)  \tag{$\left( 2.7\right) $}
\end{equation}%
By a similar argument\ we also get%
\begin{equation}
\left( \nabla _{X}P\right) Z=0;\ \ \ \ \ \ \ \ \ \ \forall \ X,\ Z\in \Gamma
\left( \mathcal{E}_{\left( -1\right) }\right)  \tag{$\left( 2.8\right) $}
\end{equation}%
But then $\ \ \left( 2.4\right) ,\ \left( 2.5\right) ,\ \left( 2.7\right) ,\ 
$and $\ \left( 2.8\right) $ \ give us that \ $\nabla P=0,$ \ i.e. \ $P$ \ is
parallel. \ The rest of the statements in \ $\left( i\right) $ \ will follow
from Lemma $\left( 2.6\right) $.

$ii):$ \ This will follow by mimicking the arguments used \ in $\left(
i\right) .$

$\left( B\right) :$ \ Now, observing that \ \ $S_{P}=-S_{\widehat{P}}$ , \ \ 
$\nabla P=\mathcal{-}\nabla \widehat{P}$ \ \ and \ \ $\nabla G=\mathcal{-}%
\frac{\sqrt{5}}{2}\nabla \widehat{P}=-\nabla \widehat{G},$ \ together with
the part \ $\left( A\right) ,$\ \ proofs of the statements \ $\left(
i\right) $ \ to \ $\left( vii\right) $ \ in part \ $B$ \ will easily follow. 
$\ \ \ \ \ \ \ \ \ \ \ \ \ \ \blacksquare $

For an almost para-Hermitian manifold \ $\ M_{P}=\left( M,\ h,\ P\right) $
and \ an almost golden Hermitian manifold \ $\ M_{G}=\left( M,\ h,\ G\right)
,$ \ ( note that, here the metric is hyperbolic with respect to the
indicated structures rather than pure) we do not have Proposition \ $(2.6/A)$
type of results. Instead, some conditions on the operator \ $S_{\varphi },$
with \ $\varphi \left( =P,\ G\right) $ \ induce some extra subclasses of
those manifolds. To be precise:

\begin{definition}
(2.5): \ 
\end{definition}

$\left( A\right) :$ \ An almost para-Hermitian manifold \ $\ \left( M,\ h,\
P\right) $ with its Levi-Civita connection \ $\nabla $\ \ is said to be, $\
\left( \left[ \mathbf{7}\right] \right) ,$

$\ i)$ \ \textit{nearly}\ \textit{para-Kaehler }if \ \ $\left( \nabla
_{X}P\right) X=0,\ \ \ \ \forall \ X\in \Gamma \left( TM\right) $

$ii)$ \ \textit{quasi}\ \textit{para-Kaehler }if \ \ $S_{P}=0.\ \ \ \ $

$\ iii)$ \ \textit{semi}\ \textit{para-Kaehler }if $\ \func{div}\left(
P\right) =0$ $\ (\ $equivalently, $\ \overset{m}{\underset{i=1}{\dsum }}%
h_{ii}S_{P}\left( e_{i},\ e_{i}\right) =0\ \ $where \ $\left\{
e_{_{1}},...,e_{m};\ Pe_{_{1}},...,Pe_{m}\right\} $ \ is a local orthonormal
frame field for \ $\Gamma \left( TM\right) $\ \ \ and \ $h_{ii}=h\left(
e_{i},\ e_{i}\right) )$

$\left( B\right) :$ \ An almost golden-Hermitian manifold \ $\ \left( M,\
h,\ G\right) $ with its Levi-Civita connection \ $\nabla $\ \ is said to be

$\ i)$ \ \textit{nearly}\ \textit{golden}-\textit{Kaehler } \ \ $\left(
\nabla _{X}G\right) X=0,\ \ \ \ \forall \ X\in \Gamma \left( TM\right) $

$ii)$ \ \textit{quasi}\ \textit{golden}-\textit{Kaehler }if \ \ $S_{\mathcal{%
R}}=0,\ \ \ \ $where \ $\mathcal{R}=P_{G},$ \ $G$-associated product
structure.

$iii)$ \ \textit{semi}\ \textit{golden}-\textit{Kaehler }if $\ \func{div}%
\left( G\right) =0.\ \ $

\section{3. Harmonicity}

\begin{definition}
$\left( 3.1\right) :$ A distribution \ $D$ \ over a $\left( semi\right) $
Riemannian manifold \ $\left( M,\ h\right) $ \ with its\ Levi-Civita
connection \ $\nabla ,$ \ is said to be \ 
\end{definition}

$\ i)$ \ $\left( c.f.\text{ }\left[ \mathbf{10}\right] \right) $ \ \textit{%
Vidal} if \ 
\begin{equation*}
\nabla _{_{X}}X\in D,\ \ \forall \ X\in \Gamma \left( D\right) .
\end{equation*}

$ii)$ \ $\left( \left[ \mathbf{1}\right] \right) \ \ $critical if

\begin{equation*}
\sum_{i=1}^{n}h_{ii}\nabla _{v_{i}}v_{i}\in D
\end{equation*}

If the restriction \ $h\mid _{D}$of \ $h$ \ to \ $D$ \ is positive (or
negative ) definite then the critical distribution \ $D$ \ is also called 
\textit{minimal}. Here \ $\left\{ v_{_{1}},...,v_{_{n}}\right\} $ \ is a
local orthonormal frame field for \ $D$ \ \ and \ \ $h_{ii}=h\left( v_{i},\
v_{i}\right) .$

\begin{remark}
$\left( 3.1\right) :$
\end{remark}

$1)$ \ For an almost product (or an almost golden) Riemannian manifold $\ \
\left( M,\ h,\ \varphi \left( =P,\ G\right) \right) $ \ 

$\ i)$ every Vidal distribution is critical.

$\ ii)$ \ if $\ \varphi $ \ is parallel then the eigendistributions \ $%
\mathcal{E}_{\left( k\right) }$ \ and \ \ $\mathcal{E}_{\left( \bar{k}%
\right) }$ \ of \ $\varphi $ \ are both Vidal and therefore they are
minimal. Here $k=1,$ $\ \bar{k}=-1$ \ for $\varphi =P$ \ and \ $k=\sigma ,$ $%
\ \bar{k}=\bar{\sigma}$ \ for $\varphi =G$

$iii)$ \ the eigendistributions \ $\mathcal{E}_{\left( 1\right) }$ \ and \ \ 
$\mathcal{E}_{\left( -1\right) }$ \ of \ $P$ \ ( $resp:$ $\mathcal{E}%
_{\left( \sigma \right) }$ \ and \ \ $\mathcal{E}_{\left( \bar{\sigma}%
\right) }$ \ of \ $G$ )\ are both minimal if and only if \ $\func{div}P=0,$
that is, $\left( M,\ h,\ P\right) $ \ is semi decomposable product
Riemannian \ ( $resp:$ \ $\func{div}G=0,$ \ that is, \ $\left( M,\ h,\
G\right) $ semi decomposable golden Riemannian) manifold

$2)$ \ For an almost product (or an almost golden) manifold $\ \ \left( M,\
h,\ \varphi \left( =P,\ G\right) \right) $ \ with a pure or hyperbolic
metric \ $h,$ \ If \ $\left\{ P,\ G\right\} $ \ is a twin pair \ then the
following are equivalent:

$i)\ \ \func{div}P=0$ \ 

$ii)\ \ \func{div}G=0$

\begin{lemma}
$\left( 3.1\right) :$ \ Let\ $\ \ F:\left( M,\ \varphi _{_{M}}\right)
\rightarrow \left( N,\ \varphi _{_{N}}\right) $ \ be a smooth map with its
differential map \ $F_{\ast }:TM\rightarrow TN,$ \ where \ $\varphi \left(
=P,\ G\right) .$
\end{lemma}

$\ i)$ \ If $\ F_{\ast }\circ P_{M}=G_{N}\circ F_{\ast }$ \ or \ \ $F_{\ast
}\circ G_{M}=P_{N}\circ F_{\ast }$ \ then \ $F$ \ is constant.

$ii)$ \ If any one of the following $\ $

\begin{itemize}
\item $F_{\ast }\circ P_{M}=\widehat{G}_{N}\circ F_{\ast },$ \ \ 

\item $F_{\ast }\circ \widehat{P}_{M}=\widehat{G}_{N}\circ F_{\ast },$ \ \ \ 

\item $F_{\ast }\circ \widehat{P}_{M}=G_{N}\circ F_{\ast },$\ \ \ \ 

\item $F_{\ast }\circ G_{M}=\widehat{P}_{N}\circ F_{\ast },$ \ \ 

\item $F_{\ast }\circ \widehat{G}_{M}=\widehat{P}_{N}\circ F_{\ast },$ \ \ 

\item $F_{\ast }\circ \widehat{G}_{M}=P_{N}\circ F_{\ast }$ \ \ 
\end{itemize}

holds then \ $F$ \ is constant

\begin{proof}
The statement $\left( i\right) \ \ $is treated in $\ \left( \ \left[ \mathbf{%
13}\right] ,\text{ \ Theorem }7\&8\right) .$ \ The argument used in \ $\left[
\mathbf{13}\right] $ \ works for all the cases \ in \ $\ \left( ii\right) \ $
\end{proof}

\begin{definition}
$\left( 3.2\right) :$ \ A smooth map \ $F:\left( M,\ \varphi _{_{M}}\right)
\rightarrow \left( N,\ \varphi _{_{N}}\right) $ \ with its differential map
\ $dF=F_{\ast }:TM\rightarrow TN$ \ is said to be \ 
\end{definition}

$\ i)$ \ $\left( \ \left[ \mathbf{2,\ 7}\right] \right) ,$\ $\left( P_{M},\
P_{N}\right) $-\textit{paraholomorphic}, $\left[ resp:\left( P_{M},\
P_{N}\right) \text{-\textit{anti-paraholomorphic}}\right] $ \ \textit{if}%
\begin{equation*}
F_{\ast }\circ P_{M}=P_{N}\circ F_{\ast },\ \ \ \ \ \ \ \ \ \ \ \ \ \left[
resp:\ F_{\ast }\circ P_{M}=\widehat{P}_{N}\circ F_{\ast }=-P\circ F_{\ast }%
\right]
\end{equation*}

$\ ii)$ \ \ 

\begin{itemize}
\item $\left( \left[ \mathbf{13}\right] \right) ,\ \left( G_{M},\
G_{N}\right) $-\textit{golden}\ \textit{if}%
\begin{equation*}
F_{\ast }\circ G_{M}=G_{N}\circ F_{\ast }
\end{equation*}

\item $\left( G_{M},\ G_{N}\right) $-\textit{antigolden if \ }%
\begin{equation*}
F_{\ast }\circ G_{M}=\widehat{G}_{N}\circ F_{\ast }=I_{N}-G_{N}F_{\ast }
\end{equation*}%
where \ $\widehat{\varphi }\left( =\widehat{P},\ \widehat{G}\right) $ \ is
the conjugate of \ $\varphi .$
\end{itemize}

We shall be writing \ $\ \pm \left( P_{M},\ P_{N}\right) $-paraholomorphic%
\textit{\ }to mean either $\left( P_{M},\ P_{N}\right) $-paraholomorphic%
\textit{\ }or\textit{\ \ }$\left( P_{M},\ P_{N}\right) $%
-anti-paraholomorphic. Similarly,\textit{\ }We shall be writing $\pm \left(
G_{M},\ G_{N}\right) $-golden to mean either $\ \left( G_{M},\ G_{N}\right) $%
-golden or $\left( G_{M},\ G_{N}\right) $-antigolden$.$

Note that since \ $\widehat{\widehat{\varphi }}=\varphi ,$ \ we have:

If a map $\ F:\left( M,\ \varphi _{_{M}}\right) \rightarrow \left( N,\
\varphi _{_{N}}\right) $ \ is $\ \left( P_{M},\ P_{N}\right) $%
-paraholomorphic\ then \ it is \ $\left( P_{M},\ \widehat{P}_{N}\right) $%
-anti-paraholomorphic as a map \ $F:\left( M,\ P_{_{M}}\right) \rightarrow
\left( N,\ \widehat{P}_{_{N}}\right) $, and if $\ \left( G_{M},\
G_{N}\right) $-golden then \ it is $\ \ \left( G_{M},\ \widehat{G}%
_{N}\right) $-antigolden as a map \ $F:\left( M,\ G_{_{M}}\right)
\rightarrow \left( N,\ \widehat{G}_{_{N}}\right) .$ \ Conversely, If a map $%
\ F:\left( M,\ \varphi _{_{M}}\right) \rightarrow \left( N,\ \varphi
_{_{N}}\right) $ \ is $\ \left( P_{M},\ P_{N}\right) $-anti-paraholomorphic\
then \ it is \ $\left( P_{M},\ \widehat{P}_{N}\right) $-paraholomorphic as a
map \ $F:\left( M,\ P_{_{M}}\right) \rightarrow \left( N,\ \widehat{P}%
_{_{N}}\right) $, and if $\ \left( G_{M},\ G_{N}\right) $-antigolden then \
it is $\ \ \left( G_{M},\ \widehat{G}_{N}\right) $-golden as a map \ $%
F:\left( M,\ G_{_{M}}\right) \rightarrow \left( N,\ \widehat{G}%
_{_{N}}\right) .$

\begin{proposition}
$\left( 3.1\right) :$ \ For twin pairs$\ \ \left\{ P_{M},\ G_{M}\right\} $ $%
\ $and $\ \left\{ P_{N},\ G_{N}\right\} $\ $\ $let$\ \ F:\left( M,\ \varphi
_{_{M}}\left( =P_{M},\ G_{M}\right) \right) \rightarrow \left( N,\ \varphi
_{_{N}}\left( =P_{N},\ G_{N}\right) \right) $ \ be a smooth map. Then the
following statements are equivalent:
\end{proposition}

$\ i)$ \ \ $F$ \ is \ $\left( P_{M},\ P_{N}\right) $-\textit{paraholomorphic}
$\ \left[ resp:\left( P_{M},\ P_{N}\right) -\text{\textit{%
anti-paraholomorphic}}\right] $\textit{.}

$\ ii)$ \ \ $F$ \ is \ $\left( G_{M},\ G_{N}\right) $-\textit{golden} $\ %
\left[ resp:\left( G_{M},\ G_{N}\right) \text{-\textit{antigolden}}\right] .$

\begin{proof}
\begin{equation*}
F\text{ \ is }\left( G_{M},\ G_{N}\right) \text{-\textit{antigolden}}
\end{equation*}

$\Leftrightarrow $%
\begin{equation*}
\ F_{\ast }\circ G_{M}=\widehat{G}_{N}\circ F_{\ast }
\end{equation*}

$\Leftrightarrow $%
\begin{equation*}
F_{\ast }\circ \left( tI+rP_{M}\right) =\left( tI+r\widehat{P}_{N}\right)
\circ F_{\ast }
\end{equation*}%
Since $\left\{ P_{M},\ G_{M}\right\} $ \ is a twin pair and then so is \ $%
\left\{ \widehat{P}_{M},\ \widehat{G}_{M}\right\} ,$ so that $%
G_{M}=tI+rP_{M} $\ \ and \ \ $\widehat{G}_{M}=tI+r\widehat{P}_{M},$ \ where
\ $t=\frac{1}{2},\ \ r=\frac{\sqrt{5}}{2},$

$\Leftrightarrow $%
\begin{equation*}
tF_{\ast }+r\left( F_{\ast }\circ P_{M}\right) =tF_{\ast }+r\left( \widehat{P%
}_{N}\circ F_{\ast }\right)
\end{equation*}

$\Leftrightarrow $%
\begin{equation*}
F_{\ast }\circ P_{M}=\widehat{P}_{N}\circ F_{\ast }
\end{equation*}

$\Leftrightarrow $

\ 
\begin{equation*}
F\ \ \ \text{is\ \ }\ \left( P_{M},\ P_{N}\right) \text{-\textit{%
anti-paraholomorphic}}
\end{equation*}%
The rest of the cases can be shown similarly. \ \ \ $\blacksquare $
\end{proof}

Let $\ F:\left( M,\ h\right) \rightarrow \left( N,\ g\right) $ be a smooth
map between \ (semi) Riemannian manifolds. \textit{The second fundamental
form } 
\begin{equation*}
\nabla F_{\ast }:\Gamma \left( TM\right) \times \Gamma \left( TM\right)
\rightarrow \Gamma \left( TN\right)
\end{equation*}%
\textit{of }\ $F$ \ is given by \ $\forall \ X,\ Y\in \Gamma \left(
TM\right) $%
\begin{equation*}
\left( \nabla F_{\ast }\right) \left( X,\ Y\right) =\ \overset{N}{\nabla }%
_{\left( F_{\ast }X\right) }\left( F_{\ast }Y\right) -F_{\ast }\left( 
\overset{M}{\nabla }_{X}Y\right)
\end{equation*}%
where \ $\overset{M}{\nabla }$ \ and \ $\overset{N}{\nabla }$ \ are\ the
Levi-Civita connections on \ $M$ \ and \ \ $N$ \ respectively. Note that the
map \ $\nabla F_{\ast }$ \ is bilinear and symmetric, $\left( see\ \left[ 
\mathbf{1,6}\right] \right) $

For a given distribution \ $D$ \ over a $\left( semi\right) $ Riemannian
manifold \ $\left( M,\ h\right) ,$ \ the \ $D$-\textit{tension field} \ \ $%
\mathcal{T}_{D}\mathcal{(}F\mathcal{)}$ \ of \ $F:\left( M,\ h\right)
\rightarrow \left( N,\ g\right) $ \ is given by $\left( c.f.\ \ \left[ 
\mathbf{1,6,\ 7}\right] \right) $ 
\begin{equation}
\mathcal{T}_{D}\mathcal{(}F\mathcal{)=}\dsum\limits_{i,j=1}^{s}h^{ij}\left(
\nabla F_{\ast }\right) \left( e_{i},\ e_{j}\right) \in \Gamma \left(
TN\right)  \tag{$\left( 3.1\right) $}
\end{equation}%
where $\ \left\{ e_{_{1}},...,e_{s}\right\} $ \ is a local frame field for \ 
$D$ \ and \ $\left( h^{ij}\right) =\left( h_{^{ij}}\right) ^{-1},\ \
h_{^{ij}}=h\left( e_{i},\ e_{j}\right) .$ \ In particular, if \ $\left\{
e_{_{1}},...,e_{s}\right\} $ \ is a local \ $h$-orthonormal frame field for
\ $D$ \ then the expression \ $\left( 3.1\right) $\ \ takes the form \ 

\begin{equation}
\mathcal{T}_{D}\mathcal{(}F\mathcal{)=}\dsum\limits_{i=1}^{s}h^{ii}\left(
\nabla F_{\ast }\right) \left( e_{i},\ e_{i}\right) \in \Gamma \left(
TN\right)  \tag{$\left( 3.2\right) $}
\end{equation}%
In the cases where \ $D=TM,$ \ \ we simply write \ $\mathcal{T(}F\mathcal{)}$
\ for $\mathcal{T}_{TM}\mathcal{(}F\mathcal{)}$\ call it the \ "\textit{%
tension field of \ }$F$\textit{\ "}

\begin{definition}
$\left( 3.3\right) :$ $\left( c.f.\ \left[ \mathbf{1,\ 6}\right] \right) $ \
A smooth map\ $\ F:\left( M,\ h\right) \rightarrow \left( N,\ g\right) $ \
is said to be harmonic \ $\left[ resp:\text{ }D\text{-harmonic}\right] $ \
if its tension field \ $\left[ resp:\text{ }D\text{-tension field}\right] $%
vanishes. In particular,
\end{definition}

\begin{itemize}
\item for a map \ $F:\left( M,\ h,\ P\right) \rightarrow \left( N,\ g\right) 
$ \ from an almost product Riemannian manifold \ $M$, if $D=$ $\overset{M}{%
\mathcal{E}}_{\left( 1\right) }\left[ resp:D=\ \overset{M}{\mathcal{E}}%
_{\left( -1\right) }\right] $ \ then \ $D$-\textit{harmonic} \textit{\ }$F$%
\textit{\ \ is also called plus-eigen harmonic }$\left[ resp:\mathit{\ }%
\text{minus-eigen harmonic}\right] $\textit{.}

\item for a map \ $F:\left( M,\ h,\ G\right) \rightarrow \left( N,\ g\right) 
$ \ from an almost golden Riemannian manifold \ $M$, if $D=$ $\overset{M}{%
\mathcal{E}}_{\left( \sigma \right) }\left[ resp:D=\ \overset{M}{\mathcal{E}}%
_{\left( \bar{\sigma}\right) }\right] $ \ then \ $D$-\textit{harmonic} 
\textit{\ }$F$\textit{\ \ is also called plus-eigen harmonic }$\left[ resp%
\text{: minus-eigenharmonic}\right] $\textit{).}
\end{itemize}

\begin{proposition}
$\left( 3.2\right) :$ \ For an almost product manifolds $\ \ \left( M,\ h,\
P\right) ,\ \ \ \left( N,\ g,\ Q\right) $ \ \ with pure or hyperbolic metric
\ $h$ $\ $with respect to \ $P$ \ and pure or hyperbolic metric $\ g$ $\ $%
with respect to $Q,$\ \ let \ $F:\left( M,\ h,\ P\right) \rightarrow \left(
N,\ g,\ Q\right) $ \ be a \ $\pm \left( P,\ Q\right) $-paraholomorphic map.
Then for every local sections \ \ $X,\ Y\in \Gamma \left( TM\right) ,$
\end{proposition}

\begin{equation}
\begin{tabular}{ll}
$\left( \nabla F_{\ast }\right) \left( PX,\ PY\right) $ & $=\left( \nabla
F_{\ast }\right) \left( X,\ Y\right) +\left( \overset{N}{\nabla }_{QX\
^{\prime }}Q\right) Y\ ^{\prime }-\left( \overset{N}{\nabla }_{Y\ ^{\prime
}}Q\right) \left( QX\ ^{\prime }\right) $ \\ 
& $-F_{\ast }\left[ \left( \overset{M}{\nabla }_{PX\ }P\right) Y-\left( 
\overset{M}{\nabla }_{Y}P\right) \left( PX\right) \right] $%
\end{tabular}
\tag{$\left( 3.3\right) $}
\end{equation}%
In particular,%
\begin{equation}
\begin{tabular}{ll}
$\left( \nabla F_{\ast }\right) \left( PX,\ PX\right) $ & $=\left( \nabla
F_{\ast }\right) \left( X,\ X\right) +S_{Q}\left( QX\ ^{\prime },X\ ^{\prime
}\right) -F_{\ast }\left[ S_{P}\left( PX\ ,X\right) \right] $ \\ 
& $=\left( \nabla F_{\ast }\right) \left( X,\ X\right) +Q\left\{ S_{Q}\left(
X\ ^{\prime },X\ ^{\prime }\right) -\lambda F_{\ast }\left[ S_{P}\left( X\
,X\right) \right] \right\} $%
\end{tabular}%
,  \tag{$\left( 3.4\right) $}
\end{equation}%
where \ $X\ ^{\prime }=F_{\ast }X,\ \ \ \ \ Y\ ^{\prime }=F_{\ast }Y,\ \ $%
and $\ $\ $\lambda =1$ \ when \ $F$ \ is $\ \left( P,\ Q\right) $%
-holomorphic, \ $\lambda =-1$ \ when \ $F$ \ is $\ \left( P,\ Q\right) $%
-antiholomorphic.

\begin{proof}
Let \ $F$ \ be a $\ \left( P,\ Q\right) $-paraholomorphic map so that \ \ $%
F_{\ast }\circ P=Q\circ F_{\ast }$ \ then%
\begin{equation*}
\begin{tabular}{ll}
$\left( \nabla F_{\ast }\right) \left( X,\ PY\right) $ & $=\ \overset{N}{%
\nabla }_{X\ ^{\prime }}\left( PY\right) ^{\prime }-F_{\ast }\left( \overset{%
M}{\nabla }_{X}\left( PY\right) \right) =\ \overset{N}{\nabla }_{X\ ^{\prime
}}Q\left( Y^{\prime }\right) -F_{\ast }\left( \overset{M}{\nabla }_{X}\left(
PY\right) \right) $ \\ 
& $=\left( \overset{N}{\nabla }_{X\ ^{\prime }}Q\right) Y\ ^{\prime
}+Q\left( \overset{N}{\nabla }_{X\ ^{\prime }}Y\ ^{\prime }\right) -F_{\ast }%
\left[ \left( \overset{M}{\nabla }_{X}P\right) Y+P\left( \overset{M}{\nabla }%
_{X}Y\right) \right] $ \\ 
& $=Q\left[ \overset{N}{\nabla }_{X\ ^{\prime }}Y\ ^{\prime }-F_{\ast
}\left( \overset{M}{\nabla }_{X}Y\right) \right] +\left( \overset{N}{\nabla }%
_{X\ ^{\prime }}Q\right) Y\ ^{\prime }-F_{\ast }\left[ \left( \overset{M}{%
\nabla }_{X}P\right) Y\right] $%
\end{tabular}%
\end{equation*}%
So we get%
\begin{equation*}
\left( \nabla F_{\ast }\right) \left( X,\ PY\right) =Q\left[ \left( \nabla
F_{\ast }\right) \left( X,\ Y\right) \right] +\left( \overset{N}{\nabla }%
_{X\ ^{\prime }}Q\right) Y\ ^{\prime }-F_{\ast }\left[ \left( \overset{M}{%
\nabla }_{X}P\right) Y\right]
\end{equation*}%
But then this gives us

\begin{equation}
\begin{tabular}{ll}
$\left( \nabla F_{\ast }\right) \left( PX,\ PY\right) $ & $=Q\left[ \left(
\nabla F_{\ast }\right) \left( PX,\ Y\right) \right] +\left( \overset{N}{%
\nabla }_{\left( PX\right) ^{\prime }}Q\right) Y\ ^{\prime }-F_{\ast }\left[
\left( \overset{M}{\nabla }_{PX}P\right) Y\right] $ \\ 
& $=Q\left[ \left( \nabla F_{\ast }\right) \left( PX,\ Y\right) \right]
+\left( \overset{N}{\nabla }_{Q\left( X\ ^{\prime }\right) }Q\right) Y\
^{\prime }-F_{\ast }\left[ \left( \overset{M}{\nabla }_{PX}P\right) Y\right] 
$%
\end{tabular}
\tag{$\left( 3.5\right) $}
\end{equation}%
and \ 
\begin{equation}
\left( \nabla F_{\ast }\right) \left( Y,\ PX\right) =Q\left[ \left( \nabla
F_{\ast }\right) \left( Y,\ X\right) \right] +\left( \overset{N}{\nabla }%
_{Y\ ^{\prime }}Q\right) X\ ^{\prime }-F_{\ast }\left[ \left( \overset{M}{%
\nabla }_{Y}P\right) X\right]  \tag{$\left( 3.6\right) $}
\end{equation}%
Using $\ \left( 3.6\right) $ \ in \ $\left( 3.5\right) $ \ and the symmetry
of \ $\nabla F_{\ast },$ \ we get%
\begin{equation}
\begin{tabular}{ll}
$\left( \nabla F_{\ast }\right) \left( PX,\ PY\right) $ & $=\left( \nabla
F_{\ast }\right) \left( X,\ Y\right) +\left( \overset{N}{\nabla }_{Q\left(
X\ ^{\prime }\right) }Q\right) Y\ ^{\prime }+Q\left( \left( \overset{N}{%
\nabla }_{Y\ ^{\prime }}Q\right) X\ ^{\prime }\right) $ \\ 
& $-F_{\ast }\left[ \left( \overset{M}{\nabla }_{PX}P\right) Y+P\left(
\left( \overset{M}{\nabla }_{Y}P\right) X\right) \right] $%
\end{tabular}
\tag{$\left( 3.7\right) $}
\end{equation}%
Finally using the fact that \ $P\circ \left( \nabla P\right) =-\left( \nabla
P\right) \circ P$ \ in the equation $\left( 3.7\right) $ \ we get \ the
desired result \ $\left( 3.3\right) .$
\end{proof}

In particular, \ since

\begin{equation*}
\begin{tabular}{l}
$S_{Q}\left( QX\ ^{\prime },X\ ^{\prime }\right) =\left( \overset{N}{\nabla }%
_{QX\ ^{\prime }}Q\right) X\ ^{\prime }-\left( \overset{N}{\nabla }_{X\
^{\prime }}Q\right) \left( QX\ ^{\prime }\right) =Q\left( S_{Q}\left( X\
^{\prime },X\ ^{\prime }\right) \right) $ \\ 
$S_{P}\left( PX\ ,X\right) =\left( \overset{M}{\nabla }_{PX\ }P\right)
X-\left( \overset{M}{\nabla }_{X}P\right) \left( PX\right) =P\left(
S_{P}\left( X\ ,X\ \right) \right) $%
\end{tabular}%
\end{equation*}

we get the equation $\ \left( 3.4\right) ,$ \ that is,%
\begin{equation*}
\begin{tabular}{ll}
$\left( \nabla F_{\ast }\right) \left( PX,\ PX\right) $ & $=\left( \nabla
F_{\ast }\right) \left( X,\ X\right) +S_{Q}\left( QX\ ^{\prime },X\ ^{\prime
}\right) -F_{\ast }\left[ S_{P}\left( PX\ ,X\right) \right] $ \\ 
& $=\left( \nabla F_{\ast }\right) \left( X,\ X\right) +Q\left\{ S_{Q}\left(
X\ ^{\prime },X\ ^{\prime }\right) -F_{\ast }\left[ S_{P}\left( X\ ,X\right) %
\right] \right\} .$%
\end{tabular}%
\end{equation*}%
For the case where \ \ $F$ \ is $\ \left( P,\ Q\right) $%
-anti-paraholomorphic, the same argument works so that the required result \ 
$\left( 3.4\right) $ \ follows. \ \ \ \ \ \ \ $\blacksquare $

For a pair of almost product $\left[ resp:\text{almost golden}\right] $
manifolds $\ \left( M,\ h,\ P\right) $, \ \ $\left( N,\ g,\ Q\right) \ \ \ %
\left[ resp:\left( M,\ h,\ G\right) ,\ \ \ \left( N,\ g,\ K\right) \right] $
\ \ with pure or hyperbolic metric \ $h$ $\ $with respect to \ $P$ \ $\left[
resp:G\ \right] $ \ and pure or hyperbolic metric $\ g$ $\ $with respect to $%
Q$\ \ $\left[ resp:K\ \right] $ \ we set the following conditions:

$\ \left( \mathbf{I}\right) :\ \ $\ $\ \left( M,\ h,\ P\right) $ \ \textit{%
is a para-Kaehler manifold and }\ $\left( N,\ g,\ Q\right) $ \ \textit{\ is
either a para-Kaehler\ manifold or a locally decomposable product Riemannian
manifold.}

$\left( \mathbf{II}\right) :\ \ $\ $\ \left( M,\ h,\ P\right) $ \ \textit{is
a locally decomposable product Riemannian manifold and} \ $\left( N,\ g,\
Q\right) $ \ \textit{\ is either a para-Kaehler manifold\ or a locally
decomposable product Riemannian manifold.}

$\left( \mathbf{III}\right) :\ \ $\ $\left( M,\ h,\ P\right) $ \ \textit{is
a quasi para-Kaehler manifold and }\ $\left( N,\ g,\ Q\right) $ \ \textit{\
is either quasi para Kaehler\ manifold or a locally decomposable product
Riemannian manifold.}

$\left( \mathbf{IV}\right) :\ $\ $\ \left( M,\ h,\ P\right) $ \ \textit{is a
locally decomposable product Riemannian manifold and} \ $\left( N,\ g,\
Q\right) $ \ \textit{\ is either a quasi para-Kaehler\ manifold or a locally
decomposable product Riemannian manifold.}

\begin{corollary}
$\left( 3.1\right) :$ \ For a map \ $F:\left( M,\ h,\ P\right) \rightarrow
\left( N,\ g,\ Q\right) $
\end{corollary}

$i)$ \ \ \textit{let \ }$F$\textit{\ \ be \ }$\pm \left( P,\ Q\right) $-%
\textit{paraholomorphic between manifolds which are holding the condition }$%
\ \left( \mathbf{I}\right) $\textit{\ \ or \ }$\left( \mathbf{II}\right) $%
\textit{\ \ then for every local section \ \ }$X,\ Y\in \Gamma \left(
TM\right) ,$%
\begin{equation*}
\left( \nabla F_{\ast }\right) \left( PX,\ PY\right) =\left( \nabla F_{\ast
}\right) \left( X,\ Y\right) \ 
\end{equation*}

$ii)$\textit{\ \ let \ }$F$\textit{\ \ be \ }$\pm \left( P,\ Q\right) $-%
\textit{paraholomorphic between manifolds which are holding the condition }$%
\ \left( \mathbf{III}\right) $\textit{\ \ or \ }$\left( \mathbf{IV}\right) $%
\textit{\ \ then for every local section \ \ }$X\in \Gamma \left( TM\right)
, $

\begin{equation}
\left( \nabla F_{\ast }\right) \left( PX,\ PX\right) =\left( \nabla F_{\ast
}\right) \left( X,\ X\right)  \tag{$\left( 3.8\right) $}
\end{equation}

\begin{proof}
:
\end{proof}

$i)$ \ Since \ $\overset{N}{\nabla }Q=0$ \ \ and \ \ $\overset{M}{\nabla }%
P=0 $ \ in the case \ $\left( \mathbf{I}\right) $ \ or \ $\left( \mathbf{II}%
\right) ,$ the result follows from Proposition \ $\left( 3.2\right) .\qquad $

$ii)$\textit{\ \ }Since \ $S_{Q}\left( QX\ ^{\prime },X\ ^{\prime }\right)
=0 $ \ \ and \ \ $S_{P}\left( PX\ ,X\right) =0$ \ in the case \ $\left( 
\mathbf{III}\right) $ \ or \ $\left( \mathbf{IV}\right) $, the result
follows from Proposition \ $\left( 3.2\right) .$ \ \ \ \ \ \ $\blacksquare $

\begin{theorem}
$\left( 3.1/A\right) :$ \ Let \ $F:\left( M,\ h,\ P\right) \rightarrow
\left( N,\ g,\ Q\right) $ \ be a \ $\pm \left( P,\ Q\right) $%
-paraholomorphic map from a semi decomposable product Riemannian manifold \ $%
M$ \ into either an \textbf{almost} product Riemannian manifold or an 
\textbf{almost }para-Hermitian manifold \ $N$ \ with Vidal
eigendistributions \ $\overset{N}{\mathcal{E}}_{\left( 1\right) }$ \ and \ \ 
$\overset{N}{\mathcal{E}}_{\left( -1\right) }$ of $\ Q.$ \ Then the
following statements are equivalent:
\end{theorem}

$i)$ \ \ $F$ \ \textit{\ is harmonic}

$ii)$ \ $F$ \ \textit{\ is\ plus}-\textit{eigen} \textit{harmonic and\
minus-eigen} \textit{harmonic}

\begin{proof}
:
\end{proof}

$(ii)\Rightarrow (i):$ \ This is obvious. \ 

$(i)\Rightarrow (ii):$ \ Let \ $\left\{ u_{_{1}},...,u_{s}\right\} $ \ \ and
\ \ $\left\{ v_{_{1}},...,v_{t}\right\} $ \ be local orthonormal frame
fields for \ \textit{\ }$\overset{M}{\mathcal{E}}_{\left( 1\right) }$ \ 
\textit{and \ }$\overset{M}{\mathcal{E}}_{\left( -1\right) }$ \
respectively. Then observe \ that

$a^{\circ })$ \ \ Since \ $\func{div}P=0$ \ and therefore \ \textit{\ }$%
\overset{M}{\mathcal{E}}_{\left( 1\right) }$ \ and\textit{\ \ }$\overset{M}{%
\mathcal{E}}_{\left( -1\right) }$ \ are both minimal distributions,%
\begin{equation*}
u=\overset{s}{\underset{i=1}{\dsum \ }}\overset{M}{\nabla }_{u_{i}}u_{i}\in
\Gamma \left( \overset{M}{\mathcal{E}}_{\left( 1\right) }\right) \text{ \ \
\ \ \ \ and \ \ \ \ \ \ \ }v=\overset{t}{\underset{i=1}{\sum \ }}\overset{M}{%
\nabla }_{v_{i}}v_{i}\in \Gamma \left( \overset{M}{\mathcal{E}}_{\left(
-1\right) }\right)
\end{equation*}

$b^{\circ })\ \ \ $Since $\ F\ $\ is$\ \ \pm \left( P,\ Q\right) $%
-paraholomorphic, $\ (a^{\circ })$ \ \ gives us$\ $%
\begin{equation*}
F_{\ast }\left( u\right) \text{ \ \ and \ }\ u_{i}^{\prime }=F_{\ast }\left(
u_{i}\right) \in \Gamma \left( \overset{N}{\mathcal{E}}_{\left( c\right)
}\right) ,\ \ \forall \ i=1,...,s
\end{equation*}%
\ \ and \ 
\begin{equation*}
F_{\ast }\left( v\right) ,\ \ v_{i}^{\prime }=F_{\ast }\left( v_{i}\right)
\in \Gamma \left( \overset{N}{\mathcal{E}}_{\left( -c\right) }\right) ,\ \
\forall \ i=1,...,t\ \ \ 
\end{equation*}

$c^{\circ })$ \ \ Since \ $N$\ \ is Vidal, \ $(b^{\circ })$ \ \ gives us%
\begin{equation*}
\overset{N}{\nabla }_{u\ _{i}^{\prime }}u\ _{i}^{\prime }\in \Gamma \left( 
\overset{N}{\mathcal{E}}_{\left( c\right) }\right) ,\text{\ }\ \ \forall \
i=1,...,s\text{\ \ \ \ \ and \ }\overset{N}{\ \ \ \nabla _{v\ _{i}^{\prime }}%
}v_{i}^{\prime }\in \Gamma \left( \overset{N}{\mathcal{E}}_{\left( -c\right)
}\right) ,\ \ \forall \ i=1,...,t\ 
\end{equation*}%
and therefore%
\begin{equation*}
\overset{s}{\underset{i=1}{\dsum \ }}\overset{N}{\nabla }_{u\ _{i}^{\prime
}}u\ _{i}^{\prime }\in \Gamma \left( \overset{N}{\mathcal{E}}_{\left(
c\right) }\right) \text{ \ \ \ \ and \ \ }\overset{t}{\underset{i=1}{\sum \ }%
}\overset{N}{\nabla }_{v\ _{i}^{\prime }}v_{i}^{\prime }\in \Gamma \left( 
\overset{N}{\mathcal{E}}_{\left( -c\right) }\right)
\end{equation*}%
So, from $\ \left( c^{\circ }\right) ,$ \ we have \ 
\begin{equation*}
\mathcal{T}_{\left( \overset{M}{\mathcal{E}}_{\left( 1\right) }\right) }%
\mathcal{(}F\mathcal{)=}\overset{s}{\underset{i=1}{\sum }}\left( \nabla
F_{\ast }\right) \left( u_{i},\ u_{i}\right) =\underset{i=1}{\overset{s}{%
\dsum }}\left[ \overset{N}{\nabla }_{u\ _{i}^{\prime }}u\ _{i}^{\prime
}-F_{\ast }\left( \overset{M}{\nabla }_{u_{i}}u_{i}\right) \right] \in
\Gamma \left( \overset{N}{\mathcal{E}}_{\left( c\right) }\right)
\end{equation*}%
and%
\begin{equation*}
\mathcal{T}_{\left( \overset{M}{\mathcal{E}}_{\left( -1\right) }\right) }%
\mathcal{(}F\mathcal{)=}\underset{i=1}{\overset{t}{\sum }}\left( \nabla
F_{\ast }\right) \left( v_{i},\ v_{i}\right) =\underset{i=1}{\overset{t}{%
\dsum }}\left[ \overset{N}{\nabla }_{v\ _{i}^{\prime }}v\ _{i}^{\prime
}-F_{\ast }\left( \overset{M}{\nabla }_{v_{i}}v_{i}\right) \right] \in
\Gamma \left( \overset{N}{\mathcal{E}}_{\left( -c\right) }\right)
\end{equation*}%
Where \ \ $c=1,$ \ when \ $F$ \ is \ $\left( P,\ Q\right) $-paraholomorphic
and \ $c=-1,$ \ when \ $F$ \ is \ $\left( P,\ Q\right) $%
-anti-paraholomorphic. But then, since \ \ \textit{\ }$\overset{N}{\mathcal{E%
}}_{\left( 1\right) }\cap \ \overset{N}{\mathcal{E}}_{\left( -1\right)
}=\left\{ 0\right\} ,$ \ \ we have that \ $\mathcal{T}_{\left( \overset{M}{%
\mathcal{E}}_{\left( 1\right) }\right) }\mathcal{(}F\mathcal{)}$ \ and \ $%
\mathcal{T}_{\left( \overset{M}{\mathcal{E}}_{\left( -1\right) }\right) }%
\mathcal{(}F\mathcal{)}$ \ are linearly independent. \ So, by the facts that 
$\ \ \ $\ $\left\{ u_{_{1}},...,u_{s};\ v_{_{1}},...,v_{t}\right\} $ \ is a
local orthonormal frame fields for $TM$ \ \ \ and \ \ 
\begin{equation*}
\mathcal{T(}F\mathcal{)=T}_{\left( \overset{M}{\mathcal{E}}_{\left( 1\right)
}\right) }\mathcal{(}F\mathcal{)}+\mathcal{T}_{\left( \overset{M}{\mathcal{E}%
}_{\left( -1\right) }\right) }\mathcal{(}F\mathcal{)},
\end{equation*}%
\ we have \ \ \ $\mathcal{T}_{\left( \overset{M}{\mathcal{E}}_{\left(
1\right) }\right) }\mathcal{(}F\mathcal{)}=0=\mathcal{T}_{\left( \overset{M}{%
\mathcal{E}}_{\left( -1\right) }\right) }\mathcal{(}F\mathcal{)}$\ \ by\ the
assumption that \ $\mathcal{T(}F\mathcal{)}=0$\ \ \ \ \ \ \ $\blacksquare $

\begin{corollary}
$\left( 3.2/A\right) :$ \ \ Let \ $F:\left( M,\ h,\ P\right) \rightarrow
\left( N,\ g,\ Q\right) $ \ be a \ $\pm \left( P,\ Q\right) $%
-paraholomorphic map from a locally decomposable product Riemannian manifold
\ $M$ \ \ into either a locally decomposable product Riemannian manifold or
nearly para-Kaehler (in particular, para-Kaehler) manifold \ $N.$ \ Then the
following statements are equivalent:
\end{corollary}

$i)$ \ \ $F$ \ \textit{\ is harmonic.}

$ii)$ \ $F$ \ \textit{\ is plus-eigen harmonic and\ minus-eigen harmonic.}

\begin{proof}
: By Remark $\left( 3.1\right) $ \ one gets that
\end{proof}

$a^{\circ })$\ \ for every locally decomposable product Riemannian manifold
\ $M,$ \ the eigendistributions \ $\overset{M}{\mathcal{E}}_{\left( 1\right)
}$ \ and \ \ $\overset{M}{\mathcal{E}}_{\left( -1\right) }$are both minimal.

$b^{\circ })$ \ \ for every nearly para-Kaehler manifold \ $N,$ \ the
eigendistributions \ $\overset{N}{\mathcal{E}}_{\left( 1\right) }$ \ and \ \ 
$\overset{N}{\mathcal{E}}_{\left( -1\right) }$ \ are also both Vidal.

So the equivalence of \ $\left( i\right) $ \ and \ $\left( ii\right) $ \ \
follows from the observations \ $\left( a^{\circ }\right) ,$ \ \ $\left(
b^{\circ }\right) $ \ and Theorem \ $\left( 2.1/A\right) .$ \ \ \ \ \ \ \ \
\ \ \ \ \ \ $\blacksquare $

\begin{theorem}
$\left( 3.1/B\right) :$ \ Let \ $F:\left( M,\ h,\ G\right) \rightarrow
\left( N,\ g,\ K\right) $ \ be a \ $\pm \left( G,\ K\right) $-golden map
from a semi decomposable golden Riemannian manifold \ $M$ \ into either an 
\textbf{almost} golden Riemannian manifold or an \textbf{almost}
golden-Hermitian manifold \ $N$ \ with Vidal eigendistributions \ $\overset{N%
}{\mathcal{E}}_{\left( \sigma \right) }$ \ and \ \ $\overset{N}{\mathcal{E}}%
_{\left( \bar{\sigma}\right) }$ of $\ K.$\ Then the following statements are
equivalent:
\end{theorem}

$i)$ \ \ $F$ \ \textit{\ is harmonic.}

$ii)$ \ $F$ \ \textit{\ is plus}-\textit{eigen} \textit{harmonic and \ minus}%
-\textit{eigen} \textit{harmonic.}

\begin{proof}
: Let \ $\ P_{G}$ \ and \ $Q_{K}$ \ denote the twin product structures of \ $%
G$ \ and \ $K$ \ \ respectively. Then by Lemma \ $\left( 2.1\right) $ \ and
Proposition $\left( 3.1\right) $ \ the hypothesis of this theorem becomes
equivalent to the hypothesis of Theorem $\left( 3.1/A\right) $, namely:
\end{proof}

"Let \ $F:\left( M,\ h,\ P_{G}\right) \rightarrow \left( N,\ g,\
Q_{K}\right) $ \ be a \ $\pm \left( P_{G},\ Q_{K}\right) $-paraholomorphic
map from a semi decomposable\ product Riemannian manifold \ $\left( M,\ h,\
P_{G}\right) $ \ into either an \textbf{almost} product Riemannian manifold
or an \textbf{almost }para-Hermitian manifold \ $\left( N,\ g,\ Q_{K}\right) 
$ \ with Vidal eigendistributions \ $\overset{N}{\mathcal{E}}_{\left(
1\right) }$ \ and \ \ $\overset{N}{\mathcal{E}}_{\left( -1\right) }$ of $\
Q_{K.}"$ \ 

So the required conclusion of the theorem follows from Theorem $\left(
3.1/A\right) $. $\ \ \ \blacksquare $

From \ Theorem $\ \left( 3.1/B\right) $ \ we immediately get

\begin{corollary}
$\left( 3.2/B\right) :$ \ Let \ $F:\left( M,\ h,\ G\right) \rightarrow
\left( N,\ g,\ K\right) $ \ be a \ $\pm \left( G,\ K\right) $-golden map
from a locally decomposable golden Riemannian manifold \ $M$ \ \ into either
a locally decomposable golden Riemannian manifold or nearly golden-Kaehler
(in particular, golden-Kaehler) manifold \ $N.$ \ Then the following
statements are equivalent:
\end{corollary}

$i)$ \ \ $F$ \ \textit{\ is harmonic.}

$ii)$ \ $F$ \ \textit{\ is plus-eigen harmonic and\ minus-eigen harmonic.\ }

\begin{proposition}
$\left( 3.3\right) :\ $\ \ \ Let \ $F:\left( M,\ h,\ P\right) \rightarrow
\left( N,\ g,\ Q\right) $ \ be a \ $\pm \left( P,\ Q\right) $%
-paraholomorphic map from an almost para-Hermitian manifold $\ \ \left( M,\
h,\ P\right) \ \ $into an almost para-Hermitian manifold or an almost
product Riemannian manifold \ $\left( N,\ g,\ Q\right) $ .Then the tension
field \ \ $\mathcal{T\ }\left( F\right) $ \ \ of \ $F$ \ \ takes the form%
\begin{equation*}
\begin{tabular}{ll}
$\mathcal{T\ }\left( F\right) $ & $=\dsum\limits_{i=1}^{m}\left\{ h\left(
e_{i},\ e_{i}\right) \left( \nabla F_{\ast }\right) \left( e_{i},\
e_{i}\right) +h\left( Pe_{i},\ Pe_{i}\right) \left( \nabla F_{\ast }\right)
\left( Pe_{i},\ Pe_{i}\right) \right\} $ \\ 
& $=-Q\left\{ \overset{m}{\underset{i=1}{\sum }}h_{ii}S_{Q}\left(
e_{i}^{\prime },\ e_{i}^{\prime }\right) -\lambda F_{\ast }\left( \func{div}%
\left( P\right) \right) \right\} $%
\end{tabular}%
\end{equation*}%
\ where \ \ $\left\{ e_{_{1}},...,e_{m},Pe_{_{1}},...,Pe_{m}\right\} $ \ \
is a local orthonormal frame field for \ $TM$ \ and \ $\lambda =1$ \ when \ $%
F$ \ is $\ \left( P,\ Q\right) $-paraholomorphic, \ $\lambda =-1$ \ when \ $%
F $ \ is $\ \left( P,\ Q\right) $-anti-paraholomorphic and \ $h_{ii}=h\left(
e_{i},\ e_{i}\right) ,$ \ $e_{i}^{\prime }=F_{\ast }\left( e_{i}\right) .$
\end{proposition}

\begin{proof}
: For an orthonormal frame field $\ \left\{
e_{_{1}},...,e_{m},Pe_{_{1}},...,Pe_{m}\right\} $ \ for \ $TM$ \ we have, by
definition,

\begin{equation}
\begin{tabular}{ll}
$\mathcal{T\ (}F\mathcal{)}$ & $\mathcal{=}\dsum\limits_{i=1}^{m}\left\{
h\left( e_{i},\ e_{i}\right) \left( \nabla F_{\ast }\right) \left( e_{i},\
e_{i}\right) +h\left( Pe_{i},\ Pe_{i}\right) \left( \nabla F_{\ast }\right)
\left( Pe_{i},\ Pe_{i}\right) \right\} $ \\ 
& $=\dsum\limits_{i=1}^{m}h_{ii}\left\{ \left( \nabla F_{\ast }\right)
\left( e_{i},\ e_{i}\right) -\left( \nabla F_{\ast }\right) \left( Pe_{i},\
Pe_{i}\right) \right\} $%
\end{tabular}
\tag{$\left( 3.9\right) $}
\end{equation}%
On the other hand, from Proposition $\left( 3.2\right) ,$ \ we have

\begin{equation*}
\left( \nabla F_{\ast }\right) \left( Pe_{i},\ Pe_{i}\right) =\left( \nabla
F_{\ast }\right) \left( e_{i},\ e_{i}\right) +Q\left\{ S_{Q}\left(
e_{i}{}^{\prime },e_{i}^{\prime }\right) -\lambda F_{\ast }\left[
S_{P}\left( e_{i}\ ,e_{i}\right) \right] \right\}
\end{equation*}%
Replacing this into \ $\left( 3.8\right) $ \ we get

\begin{equation*}
\begin{tabular}{ll}
$\mathcal{T\ (}F\mathcal{)}$ & $=-Q\left\{ \overset{m}{\underset{i=1}{\sum }}%
h_{ii}\left[ S_{Q}\left( e_{i}^{\prime },\ e_{i}^{\prime }\right) -\lambda
F_{\ast }\left( S_{P}\left( e_{i}\ ,e_{i}\right) \right) \right] \right\} $
\\ 
& $=-Q\left\{ \overset{m}{\underset{i=1}{\sum }}h_{ii}\left[ S_{Q}\left(
e_{i}^{\prime },\ e_{i}^{\prime }\right) \right] -\lambda F_{\ast }\left( 
\overset{m}{\underset{i=1}{\sum }}h_{ii}S_{P}\left( e_{i}\ ,e_{i}\right)
\right) \right\} $ \\ 
& $=-Q\left\{ \overset{m}{\underset{i=1}{\sum }}h_{ii}S_{Q}\left(
e_{i}^{\prime },\ e_{i}^{\prime }\right) -\lambda F_{\ast }\left( \func{div}%
\left( P\right) \right) \right\} $%
\end{tabular}%
\end{equation*}

\begin{theorem}
$\left( 3.2/A\right) :$ \ Let \ $F:\left( M,\ h,\ P\right) \rightarrow
\left( N,\ g,\ Q\right) $ \ be a \ $\pm \left( P,\ Q\right) $-holomorphic
map from an almost para-Hermitian manifold $\ \ \left( M,\ h,\ P\right) \ \ $%
into\ an almost para-Hermitian manifold or an almost product Riemannian
manifold \ $\left( N,\ g,\ Q\right) .$\ \ If either
\end{theorem}
\end{proof}

$i)$ \ $\left( \left[ \mathbf{2,\ 7,\ 11}\right] \right) ,\ \left( M,\ h,\
P\right) $ \ is a semi para-Kaehler manifold and \ $\left( N,\ g,\ Q\right) $
\ is a quasi\ para-Kaehler manifold,

or

$ii)$ \ \ $\left( M,\ h,\ P\right) $ \ is a semi para-Kaehler manifold and \ 
$\left( N,\ g,\ Q\right) $ \ is a\ locally decomposable product Riemannian
manifold,

then \ $F$ \ is harmonic.

\begin{proof}
:
\end{proof}

For a local orthonormal frame field $\left\{
e_{_{1}},...,e_{m},Pe_{_{1}},...,Pe_{m}\right\} $ \ for \ $TM$ \ we have, by
Proposition $\left( 3.3\right) $ \ that,

\begin{equation*}
\begin{tabular}{ll}
$\mathcal{T\ (}F\mathcal{)}$ & $=-Q\left\{ \overset{m}{\underset{i=1}{\sum }}%
h_{ii}S_{Q}\left( e_{i}^{\prime },\ e_{i}^{\prime }\right) -\lambda F_{\ast
}\left( \func{div}\left( P\right) \right) \right\} .$%
\end{tabular}%
\end{equation*}%
But then, $\lambda F_{\ast }\left( \func{div}\left( P\right) \right) =0$ \
since$\left( M,\ h,\ P\right) $ \ is semi para-Kaehler and \ $S_{Q}\left(
e_{i}^{\prime },\ e_{i}^{\prime }\right) =0$ $\ $since $\ \left( N,\ g,\
Q\right) $ \ is either quasi\ para-Kaehler or\ locally decomposable product
Riemannian. So harmonicity of \ $F$ \ follows.

\begin{theorem}
$\left( 3.2/B\right) :$ \ \ For a \ $\pm \left( G,\ K\right) $-golden map \ $%
F:\left( M,\ h,\ G\right) \rightarrow \left( N,\ g,\ K\right) $ \ from an
almost golden-Hermitian manifold $\ \ \left( M,\ h,\ G\right) \ \ $into an
almost golden-Hermitian manifold or an almost golden Riemannian manifold \ $%
\left( N,\ g,\ K\right) ,$\ \ \ if either
\end{theorem}

$i)$ \ $\left( M,\ h,\ G\right) $ \ is a semi golden-Kaehler manifold and \ $%
\left( N,\ g,\ K\right) $ \ is a quasi\ golden-Kaehler manifold,

or

$ii)$ \ \ $\left( M,\ h,\ G\right) $ \ is a semi golden-Kaehler manifold and
\ $\left( N,\ g,\ K\right) $ \ is a\ locally decomposable golden Riemannian
manifold,

then \ $F$ \ is harmonic.

\begin{proof}
: \ Let \ $\ P_{G}$ \ and \ $Q_{K}$ \ denote the twin product structures of
\ $G$ \ and \ $K$ \ \ respectively. Then by Lemma \ $\left( 2.1\right) $ \
and Remark $\left( 3.1\right) /2$ \ the hypothesis of this theorem becomes
equivalent to the hypothesis of Theorem $\left( 3.2/A\right) $, namely:
\end{proof}

"Let \ $F:\left( M,\ h,\ P_{G}\right) \rightarrow \left( N,\ g,\
Q_{K}\right) $ \ be a \ $\pm \left( P_{G},\ Q_{K}\right) $-holomorphic map
from a semi Kaehler manifold \ $\left( M,\ h,\ P_{G}\right) $ \ \ into
either a quasi\ para-Kaehler manifold or a\ locally decomposable product
Riemannian manifold."

Then the harmonicity of \ $F$ \ follows from Theorem $\left( 3.2/A\right) $. 
$\ \ $

\begin{remark}
$\left( 3.2\right) :$ \ For a non-constant map \ $F:\left( M,\ h,\ \varphi
\left( =P,G\right) \right) \rightarrow \left( N,\ g,\ \psi \left(
=Q,K\right) \right) ,$ \ when \ $h$ \ is a pure metric (with respect to $%
\varphi )$ \ the $\ \pm $\ paraholomorphicity of \ $F$ \ ( or \ $F$ \ being
a $\pm $golden ) is not much of a help for the harmonicity of \ $F$. The
best results we seem to get are Theorems $\left( 3.1\right) /A$ \ and \ $%
\left( 3.1\right) /B.$ \ On the other hand, when \ $h$ \ is hyperbolic, then 
$\pm \ $paraholomorphicity of \ $F$ \ ( or \ $F$ \ being a $\pm \ $golden )
gives its harmonicity under certain conditions as Theorems $\left(
3.2\right) /A$ \ and \ $\left( 3.2\right) /B$ \ show. On these lines we
provide the following example:
\end{remark}

\begin{example}
$\left( 3.1\right) :$ On $\ 
\mathbb{R}
^{2}$\ \ for $X=\left( x_{_{1}},\ x_{_{2}}\right) ,\ \ Y=\left( y_{_{1}},\
y_{_{2}}\right) \in \Gamma \left( T%
\mathbb{R}
^{2}\right) ,$ \ define 
\begin{equation*}
h\left( X\ ,Y\right) =\overset{2}{\underset{i=1}{\sum }}x_{i}y_{i}\text{ \ \
\ \ and \ \ }P\left( X\right) =\left( x_{_{1}},\ -x_{_{2}}\right) ,\text{ \
\ \ \ }G\left( X\right) =\left( \sigma x_{_{1}},\ \bar{\sigma}%
x_{_{2}}\right) .
\end{equation*}%
Then \ $\left( 
\mathbb{R}
^{2},\ h,\ P\right) $ \ becomes locally decomposable product Riemannian
manifold and \ $\left( 
\mathbb{R}
^{2},\ h,\ G\right) $ \ becomes locally decomposable golden Riemannian
manifold. Moreover, $\left( 
\mathbb{R}
^{2},\ h,\ P\right) $ \ and \ $\left( 
\mathbb{R}
^{2},\ h,\ G\right) $ \ are twin manifolds as $\left\{ P,\ G\right\} $ \
form a twin pair on \ $%
\mathbb{R}
^{2}.$ \ Let \ \ $f:\left( 
\mathbb{R}
^{2},\ h,\ \varphi \left( =P,G\right) \right) \rightarrow \left( 
\mathbb{R}
^{2},\ h,\ \varphi \left( =P,G\right) \right) $ \ be defined by 
\begin{equation*}
f\left( s,\ t\right) =\left( s,\ e^{t}\right) .
\end{equation*}%
Observe that
\end{example}

\begin{itemize}
\item \ $f$ \ is \ $\left( P,P\right) $-paraholomorphic and also\ $\left(
G,G\right) $-golden and yet

\item $f$ \ is not harmonic since \ 
\begin{equation*}
\mathcal{T(}f\mathcal{)=}\frac{\partial ^{2}f}{\partial s^{2}}+\frac{%
\partial ^{2}f}{\partial t^{2}}=\left( 0,\ 0\right) +\left( 0,\ e^{t}\right)
=\left( 0,\ e^{t}\right) .
\end{equation*}
\end{itemize}

\textbf{References}

\begin{description}
\item[{$\left[ 1\right] $}] Baird P, Wood. JC. \textit{Harmonic Morphism
Between Riemannian Manifolds}. London Math. Soc. Monographs New Series.
Oxford, UK: Clarendon Press 2003.

\item[{$\left[ 2\right] $}] Bejan CL, Benyounes M. \textit{Harmonic Maps
Between Para-Hermitian Manifolds}. New Developments in Differential Geometry
(Budapest 1996) 67-76. Kluwer Academic, Dordrecht.

\item[{$\left[ 3\right] $}] Crasmareanu M, Hretcanu. C. \textit{Golden
Differential Geometry}. Chaos, Solitons and Fractals. 38(2008), 1229-1238.

\item[{$\left[ 4\right] $}] Hretcanu C, Crasmareanu M. \textit{Applications
of the golden ratio on Riemannian manifolds}. Turk J Math 2009; 33: 179 -
191.

\item[{$\left[ 5\right] $}] Cruceanu V, Gadea PM, Masque JM. \textit{%
Para-Hermitian and Para-Kaehler Manifolds}. Quaderni dell'Istituto di
Matematica, Universita di Messina 1995.

\item[{$\left[ 6\right] $}] Eells J,.Lemaire L. \textit{Selected Topics in
Harmonic Maps}. CBMS Regional Conf.Series in Mathematics No 50, Amer. Math.
Soc.1983.

\item[{$\left[ 7\right] $}] Erdem S. \textit{On Almost(Para) Contact
(Hyperbolic) Metric Manifolds and Harmonicity of }$\left( \varphi ,\varphi
\right) $\textit{-Holomorphic Maps Between Them}. Houston J of Math 2002;
28: 21-45.

\item[{$\left[ 8\right] $}] Gezer A, Cengiz N, Salimov A. \textit{On
Integrability Of Golden Riemannian Structures}. Turk. J. Math 2013; 37:
693-703.

\item[{$\left[ 9\right] $}] De Leon M, Rodrigues PR.: \textit{Methods Of
Differential Geometry In Analytic Mechanics.} New York, NY, USA: Elsevier
Science Publihsers, 1989.

\item[{$\left[ 10\right] $}] Montesinos A. \textit{On Certain Classes Of
Almost Product Structures}. Michigan Math J 1983; 30: 31-36.

\item[{$\left[ 11\right] $}] Parmar VJ. \textit{Harmonic Morphisms Between
Semi-Riemannian Manifolds}. PhD, The University of Leeds, Leeds, UK, 1991.

\item[{$\left[ 12\right] $}] Salimov A, \.{I}scan M, Etayo F. \textit{%
Paraholomorphic B-Manifold and Its Properties}. Topology and Its
Applications 2007;154: 925-933.

\item[{$\left[ 13\right] $}] \c{S}ahin B, Akyol MA. \textit{Golden Maps
Between Golden Riemannian Manifolds And Constancy Of Some Maps}. Math
Commun. 2014; 19: 333-342.
\end{description}

\begin{center}
\bigskip
\end{center}

\end{document}